\documentclass[a4paper]{amsart}
\usepackage{graphicx}
\usepackage{hyperref}
\usepackage{mathrsfs}
\usepackage{cases}
\usepackage{tikz}
\usetikzlibrary{matrix,arrows,decorations.pathmorphing,decorations.markings,calc,shapes}
\usepackage{tikz-cd}

\newcounter{sarrow}
\newcommand\xrsquigarrow[1]{%
\stepcounter{sarrow}%
\mathrel{\begin{tikzpicture}[baseline= {( $ (current bounding box.south) + (0,-0.5ex) $ )}]
\node[inner sep=.5ex] (\thesarrow) {$\scriptstyle #1$};
\path[draw,<-,decorate,
  decoration={zigzag,amplitude=0.7pt,segment length=1.2mm,pre=lineto,pre length=4pt}] 
    (\thesarrow.south east) -- (\thesarrow.south west);
\end{tikzpicture}}%
}

\usepackage[a4paper,top=1in, bottom=1.25in, left=1.25in, right=1.25in]{geometry}%

\theoremstyle{plain}
  \newtheorem{thm}{Theorem}[section]
  \newtheorem{defn}{Definition}[section]
  \newtheorem{prop}{Proposition}[section]

  \newtheorem*{notation}{Notation}
\theoremstyle{definition}
  
  \newtheorem*{rem}{Remark}

\newcommand{\on}{\operatorname}

\newcommand{\g}{\mathfrak{g}}

\newcommand{\Hom}{\operatorname{Hom}}
\newcommand{\la}{\langle}
\newcommand{\ra}{\rangle}
\newcommand{\R}{\mathbb{R}}

\newcommand{\rsa}{\rightsquigarrow}
\newcommand{\xrsa}{\xrsquigarrow}

\tikzset{->-/.style={decoration={
  markings,
  mark=at position #1 with {\arrow{>}}},postaction={decorate}}}

\title{Integration of differential graded manifolds}
\author{Pavol \v Severa}
\address{Department of Mathematics, Universit\'e de Gen\`eve, Geneva, Switzerland}
\email{pavol.severa@gmail.com}

\author{Michal \v Sira\v n}
\address{Department of Theoretical Physics, Fakulta matematiky, fyziky a informatiky,
Univerzita Komensk\'eho,
Mlynsk\'a dolina F1,
842 48 Bratislava, Slovakia}
\email{miso.siran@gmail.com}

\thanks{Supported in part by  the grant MODFLAT of the European Research Council and the NCCR SwissMAP of the Swiss National Science Foundation.}

\begin{document}
\maketitle

\begin{abstract}
We consider the problem of integration of $L_\infty$-algebroids (differential non-nega\-tively graded manifolds) to $L_\infty$-groupoids.
We first construct a  ``big'' Kan simplicial manifold (Fr\'echet or Banach) whose points are solutions of a (generalized) Maurer-Cartan equation.
   The main analytic trick in our work is an integral transformation sending the solutions of the Maurer-Cartan equation to closed differential forms. 
   
   Following ideas of Ezra Getzler we then impose a gauge condition which cuts out a finite-dimensional simplicial submanifold.  This ``smaller'' simplicial manifold is (the nerve of) a  local Lie $\ell$-groupoid. The gauge condition can be imposed only locally in the base of the $L_\infty$-algebroid; the resulting local $\ell$-groupoids glue  up to a coherent homotopy, i.e. we get  a homotopy coherent diagram from the nerve of a good cover of the base to the (simplicial) category of local $\ell$-groupoids.

Finally we show that a $k$-symplectic differential non-negatively graded manifold integrates to a local $k$-symplectic Lie $\ell$-groupoid; globally these assemble to form an $A_\infty$-functor. As a particular case for $k=2$ we obtain integration of Courant algebroids.
\end{abstract}
\tableofcontents

\section{Introduction}\label{sec:prel}

Let us recall that a Lie bracket on a finite-dimensional vector space $\g$ is equivalent to a differential $Q$ on the graded-commutative algebra $\bigwedge\g^*=S((\g[1])^*)$. One possible approach to the construction of the integrating simply-connected group $G$ is due to Dennis Sullivan \cite{Sullivan}. Consider the simplicial set of morphisms of differential graded commutative algebras
$$K^\text{big}_\bullet=\text{Hom}_{dgca}\bigl((\textstyle\bigwedge\g^*, Q),(\Omega(\Delta^\bullet),d)\bigr)$$
where $\Delta^\bullet$ is the Euclidean simplex and $d$ is the de Rham differential. It can be shown that $K^\text{big}_\bullet$ is in fact a simplicial manifold and that its simplicial fundamental group is $\pi^{\text{simpl}}_1(K^\text{big}_\bullet)\cong G$. 

This can be geometrically explained as follows. Let $\xi^i$ be a basis of $\g^*$ and  $c^i_{jk}$ the structure constants of $\g$ in this basis, so that $Q\xi^i=c^i_{jk}\xi^j\xi^k/2$. An $n$-simplex $\mu\in K^\text{big}_n$ is determined by
$$A^i:=\mu(\xi^i),\ \ \ \ i=1,\ldots, \text{dim}\ \mathfrak{g},$$ 
which is a collection of $1-$forms on $\Delta^n$. Since $\mu$ respects the differentials we get
\begin{equation}\label{eq:MC}
dA^i=\frac{1}{2}c^i_{jk}A^jA^k.
\end{equation}
This is the Maurer-Cartan equation 
so $A$ is a flat $\mathfrak{g}$-connection on $\Delta^n$. In other words, we can identify the elements of $K^\text{big}_n$ with flat $\g$-connections on $\Delta^n$.

The integration of $\g$ to $G$ can be described as follows. A connection $A$ on an interval (automatically flat) will give rise to an element $g_A$ of the integrating group (the holonomy of the connection along the interval). Two such elements $g_A, g_{A'}$ are equal iff there exists a flat connection $A\tikz[baseline]\draw(0,0)to[out=-60,in=-120](1.5ex,0) (0,0)to[out=60,in=120](1.5ex,0);$ on a bigon which restricts on the boundary arcs to $A$ and $A'$ as on the left side of Figure \ref{fig:disc}. The group multiplication is given as $g_{A_1} g_{A_2}=g_{A_3}$ iff there exists a flat connection $A_\triangle$ on a triangle which restricts on the boundary in the manner of right side of Figure \ref{fig:disc}.

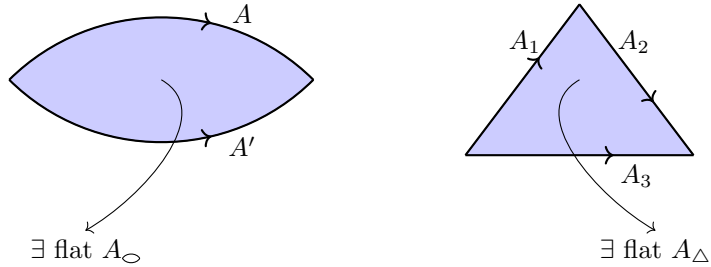
\begin{figure}[h!tb]
\center
\begin{tikzpicture}
\filldraw[fill=blue!20]
(0,0) to[out=45,in=135] (4,0) to[out=-135,in=-45] (0,0);
\draw [->-=0.65, thick] (0,0) to[out=45, in=135] node[anchor=north, near end,above]{$A$} (4,0) ;
\draw [->-=0.65, thick] (0,0) to[out=-45,in=-135] node[anchor=south, near end, below]{$A'$}(4,0);
\draw [->] (2,0) to[out=-30, in=30] (1,-2) node[anchor=south, below]{$\exists\ \text{flat}\ A\tikz[baseline]\draw[-](0,0)to[out=-60,in=-120](1.5ex,0) (0,0)to[out=60,in=120](1.5ex,0);$};
\filldraw[fill=blue!20]
(6,-1) -- (7.5,1) -- (9,-1);
\draw[->-=0.65, thick] (6,-1) -- (7.5,1) node[anchor=west,above,left,near end]{$A_1$};
\draw[->-=0.65, thick] (7.5,1) -- (9,-1) node[anchor=north,above,right, near start]{$A_2$};
\draw[->-=0.65, thick] (6,-1) -- (9,-1) node[anchor=south,below,near end]{$A_3$};
\draw[->] (7.5,0) to[out=210,in=150] (8.5,-2) node[anchor=south,below]{$\exists\ \text{flat}\ A_\triangle$};
\end{tikzpicture}
\caption{Left - two connections give the same group element. Right - group multiplication illustrated.}
\label{fig:disc}
\end{figure}

Let us note that the manifolds $K^\text{big}_n$ are infinite dimensional. For any $g\in G$ there are infinitely many $A$'s on the interval such that $g=g_A$. Likewise, if $g_{A_1} g_{A_2}=g_{A_3}$, there are infinitely many $A_\triangle$'s proving this relation. However,  $K^\text{big}_\bullet$ can be reduced by gauge-fixing to a ``smaller'' (finite dimensional) simplicial submanifold $K^s_\bullet$, which is a (local) simplicial deformation retract of $K^\text{big}_\bullet$ and which is (locally) isomorphic to the nerve of $G$ (which means a unique $A$ such that $g=g_A$, and also a unique $A_\triangle$ restricting to $A_1,A_2,A_3$ on the boundary). $K^s_\bullet$ is the nerve of the local Lie group integrating $\g$.

This paper generalizes this procedure to cases where $(\bigwedge\g^*,Q)$ is replaced by a more general differential graded-commutative algebra. More precisely, we consider generalizations involving the introduction of generators in degrees $i\geq 0$ and allowing smooth functions of the degree-0 generators rather than just polynomials.

Let us thus consider  non-negatively graded commutative algebras of the form
$$\mathcal A_V=\Gamma(S(V^*)),$$
where $V\to M$ is a negatively graded vector bundle.
Suppose that $Q$ is a differential on the algebra $\mathcal A_V$. We can say that $\mathcal A_V$ is the algebra of functions on the  graded manifold $V$.

If the manifold $M$ is a point then the differential $Q$ is equivalent to an $L_\infty$-algebra structure on the non-positively graded vector space $V[-1]$. In the case of a general $M$ and of $V$ concentrated in degree $-1$, a differential $Q$ is equivalent to a Lie algebroid structure on $V[-1]$. In the general case, the differential $Q$ is loosely called a $L_\infty$-algebroid structure on the vector bundle $V[-1]$, or more precisely a Lie $\ell$-algebroid, where  $-\ell$ is the lowest degree in the negatively graded bundle $V$ (so that $\ell\geq 1$). In this paper we will integrate these Lie $\ell$-algebroids to  Lie $\ell$-groupoids. 

Following Sullivan as above we shall study morphisms of differential graded algebras
\begin{equation}\label{eq:mor}
(\mathcal A_V,Q)\to(\Omega(\Delta^n),d),
\end{equation}
or equivalently, morphisms of differential graded manifolds
$$T[1]\Delta^n\to V.$$
We can view $L_\infty$-algebroids as Lie algebroids with higher homotopies and the integration procedure as recovering the fundamental $\infty-$groupoid. This integration procedure was suggested in \cite{Severa}. The motivation comes primarily from the problem of integration of Courant algebroids. Poisson manifolds are integrated to (local) symplectic groupoids and Courant algebroids should be integrated to symplectic 2-groupoids.

While these ideas are well known, in this work we finally overcome the long-standing analytic difficulties.
Our first result (Theorem \ref{ref:kbig}) says that  the morphisms \eqref{eq:mor} form naturally a (infinite-dimensional Fr\'echet) Kan simplicial manifold $K^{big}_\bullet$. This result is obtained via an integral transformation (similar to a transformation appearing in the work of Masatake Kuranishi \cite{Kuranishi}) taking morphisms \eqref{eq:mor} to morphisms
$$
(\mathcal A_V,0)\to(\Omega(\Delta^n),d)
$$
which is interesting on its own.

In more detail, in the simplest case when $M$ is an open subset of $\R^n$ and $V\to M$ is a trivial graded vector bundle, we choose generators $\xi^i$ of the algebra $\mathcal A_V$ ($\xi^i$'s of degree 0 are the coordinates of $M\subset \R^n$ and $\xi^i$'s of positive degree come from a basis of the fibre of $V$) then a morphism $\mu$ of dgcas \eqref{eq:mor} is equivalent to a collection of differential forms $A^i=\mu(\xi^i)\in\Omega^{\deg\xi^i}(\Delta^\bullet)$ satisfying a generalized Maurer-Cartan equation
\begin{equation}\label{eq:genMCintro}
d A^i=C_Q^i(A)
\end{equation}
where $C_Q^i:=Q\xi^i\in \mathcal A_V$ and $C_Q^i(A)=\mu(C_Q^i)\in\Omega(\Delta^\bullet)$ is obtained from $C_Q^i$ by substituting $A^j$'s for $\xi^j$'s (in the case of a Lie algebra, as discussed above, we have $C_Q^i=Q\xi^i=c^i_{jk}\xi^j\xi^k/2$). The integral transformation is $A^i\mapsto \kappa(A)^i:=A^i-hC_Q^i(A)$, where $h$ is the de Rham homotopy operator, and it transforms the equation \eqref{eq:genMCintro} to
$$d\kappa(A)^i=0.$$
We prove that $\kappa$ is an open embedding (of Banach or Fr\'echet spaces), and thus show the regularity properties of the space of the solutions of \eqref{eq:genMCintro}.

As noted in the Lie algebra case, the integration can be reduced by gauge-fixing. Following ideas of Ezra Getzler \cite{Getzler} we  define (locally in $M$) a finite-dimensional locally Kan simplicial manifold $K^s_\bullet\subset K^{big}_\bullet$ by imposing a certain gauge condition $s_\bullet$ on differential forms (see Theorem \ref{thm:ksmall})
$$K^s_\bullet=\{(A^i\in\Omega^{\deg\xi^i}(\Delta^\bullet));d A^i=C_Q^i(A)\text{ and }s_\bullet A^i=0\}.$$
The simplicial manifold $K^s_\bullet$ can be seen as (the nerve of) a local Lie $\ell$-groupoid (Definition \ref{def:l-groupoid}) integrating the Lie $\ell$-algebroid structure on $V$ ($-\ell$ is the lowest degree in $V$). While $K^s_\bullet$ depends on the choice of a gauge condition $s_\bullet$ (in particular, on a choice of local coordinates and of a local trivialization of $V$), we construct a (local) simplicial deformation retraction  of $K^{big}_\bullet$ to $K^s_\bullet$, i.e.\ show that $K^{big}_\bullet$ and $K^s_\bullet$ are equivalent as local Lie $\infty$-groupoids. The deformation retraction also implies that $K^s_\bullet$ is unique up to (non-unique) isomorphisms and that the local $K^s_\bullet$'s form a homotopy coherent diagram.

Let us relate our approach with the papers of Henriques \cite{Henriques} and Getzler \cite{Getzler}, who solve closely related problems. Andr\'e Henriques deals in \cite{Henriques} with the case when the base manifold $M$ is a point. He defined Kan simplicial manifolds and Lie $\ell$-groupoids, which are central definitions in our work. His constructions are based on Postnikov towers: in his case it is enough to integrate a Lie algebra to a Lie group and then to deal with Lie algebra cocycles. This approach cannot be used for non-trivial $M$, so there is little overlap between his and our methods.

The present paper is  closer to the work of Ezra Getzler \cite{Getzler} who deals with nilpotent $L_\infty$-algebras (or, from the point of view of Rational homotopy theory \cite{Sullivan}, with finitely-generated Sullivan algebras). In particular, the idea of gauge fixing is simply taken from \cite{Getzler} and translated from formal power series to the language of Banach manifolds. 

Our paper is also closely related to the work of Crainic and Fernandes \cite{CrainicFernandes} on integration of Lie algebroids (corresponding to $\ell=1$). Unlike in \emph{op.\ cit.}\ we do not consider the truncation of $K^{big}_\bullet$ at dimension $\ell$ (keeping all simplices of dimension $<\ell$ intact, replacing those of dimension $\ell$ with their homotopy classes rel boundary, and adding higher simplices formally), as it leads, for $\ell\geq2$, to infinite-dimensional spaces. Nonetheless our analytic results should be sufficient also for this kind of approach.

The plan of our paper is as follows. In Section \ref{sec:dg-mor} we recall some basic definitions concerning NQ manifolds (differential non-negatively graded manifolds). In Section \ref{sec:homotopy} we prove a technical result about homotopies of maps between NQ manifolds, which is the basis for our analytic theorems. In Section \ref{sec:kuranishi} we define the Kuranishi map $\kappa$ and show that it transforms the Maurer-Cartan equations \eqref{eq:genMCintro} to linear equations. In this way we then prove that the spaces of solutions are manifolds. In Section \ref{sec:fh} we prove that the spaces of the solutions of the generalized Maurer-Cartan \eqref{eq:genMCintro} on simplices form a Kan simplicial manifold (i.e.\ a Lie $\infty$-groupoid) and in Section \ref{sec:glob} we globalize the results of Sections \ref{sec:kuranishi} and \ref{sec:fh}. (Sections \ref{sec:fh} and \ref{sec:glob} are not needed for the rest of the paper, but they complete the picture.) In Section \ref{sec:gauge} we show that gauge fixing produces a finite-dimensional local Lie $\ell$-groupoid. In Section \ref{sec:def-retr} we establish a deformation retraction from the big integration to the gauge integration, which is then used in Section \ref{sec:funct} to show that the gauge integration is functorial up to coherent homotopies. Section \ref{sec:sympl} concerns integration of (pre)symplectic forms on a NQ manifold. Closing the paper with Section \ref{sec:sympl-funct} we describe the $A_\infty$-functoriality of the local integration of symplectic NQ manifolds; as an example we describe the integration of Courant algebroids.

\subsection*{Acknowledgment}
We would like to thank Ezra Getzler for useful discussion.

\section{dg morphisms and NQ manifolds}\label{sec:dg-mor}
Let $M$ be a manifold and $V\to M$ a negatively-graded vector bundle. Throughout the paper all vector bundles are finite dimensional.  Let $\mathcal A_V$ be the graded commutative algebra
$$\mathcal A_V=\Gamma(S(V^*)).$$
In particular $\Omega(M)=\mathcal A_{T[1]M}$. We can describe morphisms of graded-commutative algebras
$$\mathcal A_V\to\Omega(N)$$
(where $N$ is another manifold) in the following way.

If $W$ is a non-positively graded vector space and $N$ a manifold, let
\begin{equation}\label{eq:omegaw}
\Omega(N,W)^m={\textstyle\bigoplus_{k\geq0}}\Omega^{m+k}(N)\otimes W^{-k}.
\end{equation}
Abusing the notation, if $V\to M$ is a negatively graded vector bundle, let
$$\Omega(N,V)^0=\{(f,\alpha);\,f:N\to M, \alpha\in{\textstyle\bigoplus_{k>0}}\Omega^k(N,f^*V^{-k})\}.$$
The two notations are compatible if we understand $W$ as a trivial vector bundle over $W^0$, with the fiber $W^{<0}$.

A morphism of graded algebras
$$\mu:\mathcal A_V\to\Omega(N)$$
is equivalent to an element  $(f,\alpha)\in\Omega(N,V)^0$ via
$$\mu|_{C^\infty(M)}=f^*,\ \mu(s)=\langle \alpha,f^*s\rangle, \ \ \ \ \forall s\in\Gamma(V^*).$$

If $W$ is a non-positively graded vector space and $U\subset W^0$ is an open subset, let $W|_U$ be the trivial vector bundle $W^{<0}\times U\to U$, i.e.
$$\Omega(N,W|_U)^0\subset\Omega(N,W)^0$$
is the set of those forms $A\in\Omega(N,W)^0$ whose function part (a $W^0$-valued function on $N$) is a map $N\to U\subset W^0$. Notice that
\begin{equation}\label{eq:algloc}
\mathcal A_{W|_U}=C^\infty(U)\otimes S\bigl((W^{<0})^*\bigr).
\end{equation}
In this special case of $V=W|_U$ a morphism $\mu:\mathcal A_{W|_U}\to\Omega(N)$ corresponds to $A\in\Omega(N,W|_U)^0$
via
$$\mu(\xi)=\langle A,\xi\rangle, \ \ \ \ \forall \xi\in W^*.$$

Let now $Q$ be a differential on the graded algebra $\mathcal A_{W|_U}$. If $\xi^i$ is a (homogeneous) basis of $W^*$, let 
$$C^i_Q:=Q\,\xi^i\in\mathcal A_{W|_U}.$$
 The identity $Q^2=0$ is equivalent to
\begin{equation}\label{eq:Q2coord}
C^i_Q\frac{\partial C^k_Q}{\partial \xi^i}=0.
\end{equation}

A morphism of graded algebras $\mu:\mathcal A_{W|_U}\to\Omega(N)$ is a differential graded (dg) morphism iff
\begin{equation}\label{eq:dgmorph}
dA^i=C^i_Q(A)
\end{equation}
where $A^i:=\mu(\xi^i)=\langle A,\xi^i\rangle$ and $C^i_Q(A)$ is obtained from $C^i_Q$ by substituting $A^k$'s for $\xi^k$'s. This equation generalizes the Maurer-Cartan equations of the Lie algebra case (\ref{eq:MC}) where $C^i_Q$ is quadratic.
  We can rewrite \eqref{eq:dgmorph} as
\begin{equation}\label{eq:genmc}
dA=C_Q(A)
\end{equation}
where $C_Q\in  \mathcal A_{W|_U}\otimes W$ is given by 
$$\la C_Q,\xi\ra=Q\xi\ \ \ \ \forall \xi\in W^*.$$

\begin{rem}
If $K$ is another manifold and $s:K\to N$ is a smooth map then $s^*:\Omega(N)\to\Omega(K)$ is a morphism of dg algebras, and so if $\mu:\mathcal A_{W|_U}\to\Omega(N)$ is a dg morphism then $s^*\circ\mu:\mathcal A_{W|_U}\to\Omega(K)$ is a dg morphism. Equivalently, if $dA=C_Q(A)$ then $d(s^*A)=C_Q(s^*A)$. It can also be seen from the obvious identity $s^*C_Q(A)=C_Q(s^*A)$.
\end{rem}

The set of solutions $A$ of \eqref{eq:genmc} will be denoted by
$$\Omega(N,W|_U)^{MC}\subset\Omega(N,W|_U)^0.$$
More generally, if $Q$ is a differential on the algebra $\mathcal A_V$, the set of dg morphisms
$$\mathcal A_V\to\Omega(N)$$
will be denoted by 
$$\Omega(N,V)^{MC}\subset\Omega(N,V)^0.$$

\begin{rem}
We have a natural inclusion $M\subset \Omega(N,V)^{MC}$ (and $U\subset\Omega(N,W|_U)^{MC}$) given by pairs $(f,\alpha)$ where $f:N\to M$ is a constant map and $\alpha=0$. Indeed, these pairs correspond to the (somewhat trivial) dg morphisms $\mathcal A_V\to\Omega(N)$ factoring through $\R$.
\end{rem}

If $V\to M$ is a negatively-graded vector bundle, it is convenient to see the algebra $\mathcal A_V$ as the algebra of functions on the graded manifold (corresponding to) $V$. 

\begin{defn}[e.g. \v Severa \cite{Severa}]
An \emph{N-manifold} (shorthand for \emph{non-negatively graded manifold}) is a supermanifold $Z$ with an action of the semigroup $(\R,\times)$ such that $-1\in\R$ acts as the parity operator (i.e.\ just changes the sign of the odd coordinates).
\end{defn}

 Let $C^\infty(Z)^k$ be the vector space of smooth functions of degree $k$, where the degree is the weight with respect to the action of $(\R,\times)$, and let $C^\infty(Z)=\bigoplus_{k\geq 0} C^\infty(Z)^k$. It is a graded-commutative algebra. 

Any negatively-graded vector bundle $V\to M$ gives rise to a N-manifold via
$$C^\infty(Z)=\Gamma(S(V^*)),$$
such that $0\cdot Z=M$. Any N-manifold (in the $C^\infty$-category) is of this type. By abuse of notation we shall denote the N-manifold corresponding to $V$ also by $V$. A morphism of graded algebras $\mathcal A_V\to\mathcal A_{V'}$ is thus equivalent to a morphism of N-manifolds $V'\rsa V$. We shall indicate morphisms of graded manifolds with wiggly arrows ($\rsa$). This is to prevent confusion since the category of N-manifolds contains  more morphisms then the category of negatively-graded vector bundles.
In particular, the following  are equivalent: a morphism of graded algebras
$$\mathcal A_V\to\Omega(N),$$
an element of $\Omega(N,V)^0$, and a morphism of N-manifolds
$$T[1]N\rsa V.$$ 

A differential $Q$ on the graded algebra $\mathcal A_V$ then corresponds to the following notion.

\begin{defn}
An \emph{NQ-manifold} (a \emph{differential non-negatively graded manifold}) is an $N$-manifold equipped with a vector field $Q$ of degree 1 satisfying $Q^2=0$.
\end{defn}

In particular, a morphism of dg-algebras $\mathcal A_V\to\Omega(N)$ is equivalent to a morphism of NQ-manifolds $T[1]N\rsa V$.

\section{A homotopy lemma}\label{sec:homotopy}

Let us suppose, as above, that $W$ is a non-positively graded finite-dimensional vector space, $U\subset W^0$ is an open subset and  $Q$ a differential on the algebra $\mathcal A_{W|_U}$ defined by \eqref{eq:algloc}. A morphism of graded algebras $\mathcal A_{W|_U}\to\Omega(N)$ (or a map of N-manifolds $T[1]N\rsa W|_U$) is thus equivalent to a choice of a differential form $A\in\Omega(N,W|_U)^0$, and it is a dg morphism iff $dA=C_Q(A)$, i.e. iff $A\in\Omega(N,W|_U)^{MC}$ (see Section \ref{sec:dg-mor}).

In this section we shall study ``homotopies'', i.e.\ solutions of the MC equation $dA=C_Q(A)$ on $N\times I$, where $I$ is the unit interval.

If $\alpha\in\Omega(N\times I)$, let $\alpha|_{t=0}\in\Omega(N)$ (where $t$ is the coordinate on $I=[0,1]$) be the restriction of $\alpha$ to $N=N\times\{0\}\subset N\times I$.  Any differential form $\alpha\in\Omega(N\times I)$ can be split uniquely into a horizontal and vertical part as
$$\alpha=\alpha_h+dt\,\alpha_v$$
$$i_{\partial_t}\alpha_h=i_{\partial_t}\alpha_v=0.$$
The components $\alpha_h$ and $\alpha_v$ are given by 
$$\alpha_v=i_{\partial_t}\alpha,\ \alpha_h=\alpha-dt\,\alpha_v.$$
We can view $\alpha_h$ and $\alpha_v$ as forms on $N$ parametrized by $I$.
For any $\alpha$ we set
$$d_h\alpha=(d\alpha)_h.$$

\begin{prop}[``Homotopy Lemma'']\label{prop:htopy}
Let $N$ be a manifold (possibly with corners) and let $A\in\Omega(N\times I,W|_U)^0$, where $I$ is the unit interval.
If
\begin{subequations}
\begin{equation}\label{eq:Acond1}
A|_{t=0}\in\Omega(N,W|_U)^{MC}
\end{equation}
and if
\begin{equation}\label{eq:Acond2}
i_{\partial_t}(dA-C_Q(A))=0,
\end{equation}
\end{subequations}
then 
$$A\in\Omega(N\times I,W|_U)^{MC}$$
\end{prop}

\begin{figure}[h!tb]
\center
\begin{tikzpicture}
\filldraw[fill=blue!20]
node[anchor=west,left,pos=0]{$N\times \{t=0\}$} (0,0) --  node[anchor=west,left,pos=0.5]{$N\times I$} node[anchor=west,left,pos=1]{$N\times \{t=1\}$} (0,3) -- (1,2.5) -- (1,-0.5) -- (0,0);
\filldraw[fill=blue!20]
(2,0) -- (2,3) -- (1,2.5) -- (1,-0.5) -- (2,0);
\filldraw[fill=blue!10]
(0,3) -- (2,3) -- (1,2.5);
\draw[dashed] (0,0) -- (2,0);
\draw[->] (2,-0.25) to[out=270,in=180] (3,-0.5) node[anchor=east,right]{$A|_{t=0}\in\Omega(N,W)^{MC}$};
\draw[->] (1.5,1.5) to[out=270, in=180] (3,1) node[anchor=east,right]{$i_{\partial_t}(dA-C_Q(A))=0$};
\draw (5,2.5)node{$\Rightarrow A\in\Omega(N\times I,W)^{MC}$};
\end{tikzpicture}
\caption{Homotopy Lemma. }
\label{fig:diagram}
\end{figure}
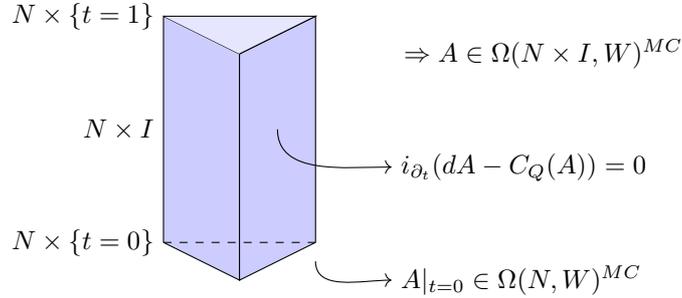

\begin{proof}

Writing $A=A_h+dt\,A_v$, and using the Taylor expansion
$$C_Q(A)=C_Q(A_h+dt\,A_v)=C_Q(A_h)+dt\,A_v^i\frac{\partial C_Q}{\partial \xi^i}(A_h)$$
we get
\begin{equation}\label{eq:decdA}
dA-C_Q(A)=(d_h A_h-C_Q(A_h))+dt\Bigl(\frac{d}{dt}A_h-d_hA_v-A_v^i\frac{\partial C_Q}{\partial \xi^i}(A_h)\Bigr).
\end{equation}
Since $i_{\partial_t}(dA-C_Q(A))=0$, we get from here
\begin{equation}\label{eq:dotAh}
\frac{d}{dt}A_h=d_hA_v+A_v^i\frac{\partial C_Q}{\partial \xi^i}(A_h).
\end{equation}

We can now compute
$$\frac{d}{dt}(d_hA_h-C_Q(A_h))$$
using \eqref{eq:dotAh}. We get
$$\frac{d}{dt}(d_hA_h-C_Q(A_h))=
(-1)^{\deg \xi^i-1}A_v^i\,d_hA_h^k\frac{\partial^2 C_Q}{\partial \xi^k \partial \xi^i}(A_h)
-A^i_v\frac{\partial C^k_Q}{\partial \xi^i}(A_h)\frac{\partial C_Q}{\partial \xi^k}(A_h).$$
Since Equation \eqref{eq:Q2coord} implies
$$0=\frac{\partial}{\partial \xi^i}\Bigl(C^k_Q(A_h)\frac{\partial C_Q}{\partial \xi^k}(A_h)\Bigr)=
\frac{\partial C^k_Q}{\partial \xi^i}(A_h)\frac{\partial C_Q}{\partial \xi^k}(A_h)+(-1)^{\deg \xi^i}C^k_Q(A_h)\frac{\partial^2 C_Q}{\partial \xi^k \partial \xi^i}(A_h),
$$
we get 
$$\frac{d}{dt}(d_hA_h-C_Q(A_h))=(-1)^{\deg \xi^i-1}A_v^i(d_hA_h^k-C^k_Q(A_h))\frac{\partial^2 C_Q}{\partial \xi^k \partial \xi^i}(A_h).$$
This is a linear differential equation for $d_hA_h-C_Q(A_h)$, which together with the initial condition \eqref{eq:Acond1} implies that 
$$d_hA_h-C_Q(A_h)=0$$
and thus, in view of \eqref{eq:Acond2}, also
$$dA-C_Q(A)=0.$$
This completes the proof.
\end{proof}

\begin{thm}\label{thm:htopy2}
Let $N$ be a compact manifold (possibly with corners), let $A_0\in\Omega(N,W|_U)^{MC}$, and let $H\in\Omega(N\times I,W)^{-1}$ be a horizonal form (i.e.\ $i_{\partial_t}H=0$).
Then there is $0<\epsilon\leq1$ and a unique $A\in\Omega(N\times [0,\epsilon],W|_U)^{MC}$ such that $A|_{t=0}=A_0$ and $A_v=H$. It is given by $A=A_h+dt\,H$, where $A_h$ is the solution of the differential equation
\begin{equation}\label{eq:dotAh2}
\frac{d}{dt}A_h=d_hH+H^i\frac{\partial C_Q}{\partial \xi^i}(A_h)
\end{equation}
 with the initial condition $A_h|_{t=0}=A_0$. We can choose $\epsilon=1$ if the $C^0$-norm of $H^{(0)}$ is small enough, where $H^{(0)}:N\times I\to W^{-1}$ is the 0-form part of $H$.
\end{thm}
\begin{proof}
By Proposition \ref{prop:htopy} and Equation \eqref{eq:decdA} we see that $A\in\Omega(N\times I,W|_U)^{MC}$ iff the condition \eqref{eq:dotAh} holds with $A_v=H$, i.e.\ if the ODE \eqref{eq:dotAh2} holds. 

To see that \eqref{eq:dotAh2} admits a (unique) solution for a small-enough $\epsilon$, notice that \eqref{eq:dotAh2} is a family of finite-dimensional ODEs parametrized by $N$. Namely, for each $x\in N$  it is a ODE for $A_h|_x\in(\bigwedge T^*_x N\otimes W)^1$. If we take the solutions of these ODEs with the initial values $A_0|_x$ on the maximal subintervals $J_x\subset I$, we get the existence and uniqueness of a solution $A$ on the open subset $O\subset N\times I$, $O=\bigcup_{x\in N}\{x\}\times J_x$. As $N\times\{0\}\subset U$, the compactness of $N$  implies the existence of $\epsilon>0$ s.t.\ $N\times[0,\epsilon]\subset O$.

 To see that an estimate of the $C^0$-norm of $H^{(0)}$ implies that \eqref{eq:dotAh2} has a solution for $t\in[0,1]$, let us split $A_h$ to its homogeneous parts
$$A_h=\sum_k A_h^{(k)},\quad A_h^{(k)}\in\Omega^k(N\times I,W^{-k}).$$ 
In degree 0 Equation \eqref{eq:dotAh2} is
\begin{equation}\label{eq:dotAh0}
\frac{d}{dt}A_h^{(0)}=(H^{(0)})^i\frac{\partial C_Q}{\partial \xi^i}(A_h^{(0)})
\end{equation}
with $i$ running only over the indices with $\deg\xi^i=1$. Again, this is a family of ODEs for $A_h|_x\in U\subset W^0$, parametrized by $x\in N$, and the smallness of the $C^0$-norm of $H^{(0)}$ implies that the RHSs are small uniformly in $x$ and thus the solution exists on $N\times I$. Supposing that $A_h^{(m)}$'s are known for $m<k$, Equation \eqref{eq:dotAh2} in degree $k$ is an inhomogeneous linear ODE for $A_h^{(k)}$, and therefore always has a solution. We can thus find $A_h$ by solving \eqref{eq:dotAh2} successively for $A_h^{(0)}$, $A_h^{(1)}$, \dots, $A_h^{(\dim N)}$.
\end{proof}

\begin{rem}
The hypothesis on the $C^0$-norm of $H^{(0)}$ was used to make sure that a solution $A_h^{(0)}$ of Equation \eqref{eq:dotAh0} exists for $t\in I$. If $W^0=0$, no hypothesis is needed, since $A_h^{(0)}=0$. When $W^0\neq0$ then $Q:C^\infty(U)
\to C^\infty(U)\otimes (W^{-1})^*$ can be seen as a linear map from $W^{-1}$ to the space of vector fields on $C^\infty(U)$, and defines an integrable distribution on $U$. If $A_0\in\Omega(N,W|_U)^{MC}$ then the image of $A_0^{(0)}:N\to U$ is contained in a leaf of this distribution. If the leaf if compact, again no hypothesis of the $C^0$-norm of $H^{(0)}$ is needed, since \eqref{eq:dotAh0} is given by vector fields tangent to the leaf. The hypothesis is needed only if the closure of the leaf is non-compact. 
\end{rem}

Theorem \ref{thm:htopy2} can be used to solve the generalized Maurer-Cartan equation $dA=C_Q(A)$ on cubes. We shall describe another method in the following section.

\section{Solving the generalized Maurer-Cartan equation}\label{sec:kuranishi}
Suppose now that $N\subset\R^n$ is a star-shaped $n$-dimensional  submanifold with corners (typically we would take for $N$ a $n$-simplex with a vertex at the origin, or a ball centered at the origin). Let 
$$h:\Omega^\bullet(N)\to\Omega^{\bullet-1}(N)$$
 be the de Rham homotopy operator given by the deformation retraction $$R:N\times I\to N,\ R(x,t)=tx$$
 (i.e.\ $h\alpha=\int_I R^*\alpha$).
Let, as above, $U\subset W^0$ be an open subset and $Q$ a differential on the algebra $\mathcal A_{W|_U}$.

\begin{thm}\label{thm:extiso}
A form $A\in\Omega(N,W|_U)^0$ satisfies
\begin{equation}\label{eq:mcq}
dA=C_Q(A)
\end{equation}
if and only if the form 
$$B=A-h(C_Q(A))\in\Omega(N,W)^0$$
satisfies
\begin{equation}\label{eq:mc0}
dB=0.
\end{equation}
\end{thm}
\begin{proof}
If $dA=C_Q(A)$ then $dhC_Q(A)=dhdA=dA$, hence $B$ is closed. 

Suppose now that $dB=0$. Let $E$ be the Euler vector field on $N\subset\R^n$. By construction of $h$ we have $i_Eh=0$, hence $i_Edh=i_E$, and thus 
$$i_E (dA-C_Q(A))=i_E(dA-dhC_Q(A))=i_E dB=0.$$
Since the vector field $E$ on $N$ and the vector field $t\partial_t$ on $N\times I$ are $R$-related, we get
$$i_{\partial_t}(d\,R^*A-C_Q(R^*A))=0$$
(using $R^*C_Q(A)=C_Q(R^*A)$).
As $(R^*A)|_{t=0}$ is a constant we have $(R^*A)|_{t=0}\in\Omega(N,W|_U)^{MC}$. Proposition \ref{prop:htopy} now implies that $d\,R^*A-C_Q(R^*A)=0$, and therefore (by setting $t=1$) $dA=C_Q(A)$.
\end{proof}

Let us  show that the map $A\mapsto B= A-h(C_Q(A))$ is injective by describing explicitly its inverse.

\begin{prop}\label{prop:K-inv}
Let $B\in\Omega(N,W)^0$ be such that the 0-form part of $B$ (a map $N\to W^0$) sends $0\in N\subset\R^n$ to an element of $U\subset W^0$. The equation
\begin{equation}\label{eq:a}
\mathcal{L}_E a=C_Q(B+i_E a)
\end{equation}
($\mathcal{L}_E$ is Lie derivative along $E$) has a solution $a\in\Omega(N',W)^1$ on $N'\subset N$ where $N'$ is some star-shaped open neighborhood of $0$. The solution $a$ is unique and we can thus demand $N'$ to be maximal (i.e. containing any other $N'$). 

A form  $A\in\Omega(N,W|_U)^0$ such that 
$$B=A-h(C_Q(A))$$ exists iff $N'_{max}=N$, and it is unique, namely $A=B+i_Ea$.
\end{prop}
\begin{proof}
We shall define a vector field $\hat E$ on the total space of 
$$\hat N:=(\textstyle\bigwedge T^*N\otimes W)^1$$
with the following properties:
\begin{enumerate}
\item A (partial) section $a:N'\to\hat N$ ($N'\subset N$) of the bundle $\hat N\to N$ is a solution of \eqref{eq:a} iff the vector field $\hat E$ is tangent to the image of $a$. This simply reflects the fact that \eqref{eq:a} is a quasi-linear PDE.
\item $\hat E$ projects to the Euler vector field $E$ on $N$.
\item $\hat E$ has a unique fixed point $P\in\hat N$ (lying over $0\in N$). The fixed point is hyperbolic: the stable subspace of $T_P\hat N$ is the vertical subspace, and the unstable subspace  projects bijectively onto $T_0N$.
\end{enumerate}

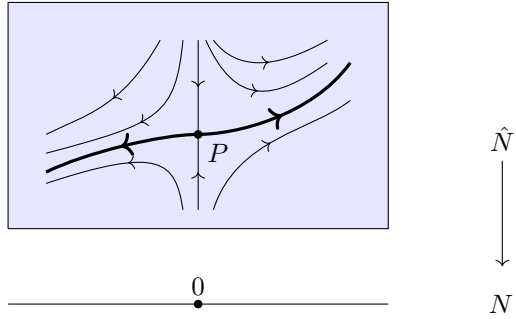
\begin{figure}[h!tb]
\center

\begin{tikzpicture}
{\filldraw[fill=blue!10] (0,0) -- (5,0) --  (5,3) -- (0,3) -- (0,0);
\draw[->] (6.5,0.9)node[anchor=north, above]{$\hat N$} -- (6.5,-0.5);
\filldraw (2.5,-1) circle (.05cm) node[anchor=north,above]{$0$};
\filldraw (2.5,1.25) circle (.05cm);
\draw (0,-1) -- node[anchor=east,right,pos=1.24]{$N$} (5,-1);
\draw[->-=.5] (2.5,2.5) -- (2.5,1.25);
\draw[->-=.5] (2.5,0.25) -- node[anchor=east,right,pos=0.75]{$P$} (2.5,1.25);
\draw[very thick,->-=.5] (2.5,1.25) .. controls (2,1.25)  and (1,1) .. (0.5,0.75);
\draw[very thick,->-=.5] (2.5,1.25) .. controls (3,1.25) and (4,1.5) .. (4.5,2.2);
\draw[->-=.5] (2.3,2.5) .. controls (2.2,1.5) and (2,1.4) .. (0.5,1);
\draw[->-=.5] (2,2.5) .. controls (1.5,1.6) and (1,1.5) .. (0.5,1.25);
\draw[->-=.5] (2.7,2.5) .. controls (3,2.1) and (3.5,2.1) .. (4.2,2.5);
\draw[->-=.5] (2.6,2.5) .. controls (3,1.6) and (3.5,1.7) .. (4.2,2.2);
\draw[->-=.5] (2.3,0.25) .. controls (2.2,0.9) and (2,1.1) .. (0.5,0.6);
\draw[->-=.5] (2.7,0.25) .. controls (3,1.2) and (4,1.3) .. (4.5,1.7);
}
\end{tikzpicture}

\caption{The section $a$ is the unstable manifold of $\hat E$}

\label{fig:hyperbolic}
\end{figure}

Once $\hat E$ is defined, the proposition can be proven as follows. A partial section $a:N'\to \hat N$, where $N'\subset N$ is a star-shaped neighbourhood of $0\in N$, is a solution of \eqref{eq:a} iff the image of $a$ is a local unstable manifold of $\hat E$. Since $\hat E$ projects onto $E$, the (full) unstable manifold of $\hat E$ is the image of some section $a:N'_{max}\to \hat N$, and this section is the unique maximal solution of \eqref{eq:a} we wanted to find. 

The existence and uniqueness of $A$ can then be proven as follows. If $N'_{max}=N$, i.e.\ if $a$ is defined on the entire $N$, then we can set $A=B+i_Ea$ and it satisfies $B=A-h(C_Q(A))$, as required. To get uniqueness, notice that the operator $\mathcal{L}_E$ is invertible on $\Omega^{>0}(N)$  and $h=i_E \mathcal{L}_E^{-1}$.\footnote{The inverse of $\mathcal L_E$ be written explicitly as an follows: if $\alpha=\sum_I f_Idx^I\in\Omega^k(N)$ ($I$ is a multiindex) then $\mathcal L_E^{-1}\alpha = \sum_I g_Idx^I$ with $g_I(x)=\int_0^1 t^{k-1}f_I(tx)\,dt$. The identity $h=i_E \mathcal{L}_E^{-1}$ follows from $h i_E =0$ since $h L_E= h d i_E + h i_E d = h d i_E = (h d + d h) i_E = i_E$.} 
If $A$ satisfying $B=A-h(C_Q(A))$ exists, let us set $a=\mathcal{L}_E^{-1}C_Q(A)$. As 
$$A=B+h \mathcal{L}_E a=B+i_E a,$$
 $a$ is a solution of \eqref{eq:a} with $N'=N$. The uniqueness of $A$ then follows from the uniqueness of $a$.

The vector field $\hat E$ on $\hat N$ is constructed as follows. Consider the natural lift $E_\text{lift}$ of $E$ to $\hat N$. In coordinates it is of the form
$$E_\text{lift}=x^i\frac{\partial}{\partial x^i}-\sum_{(J,a)} k_{(J,a)} y^{(J,a)} \frac{\partial}{\partial y^{(J,a)}}$$
where $x^i$'s are the coordinates on $N$ and $y^{(J,a)}$ are the additional coordinates on $\hat N$, dual to $dx^{j_1}\ldots dx^{j_l}\otimes e_a$ (where $e^a$ is a basis of $W$ and $\deg e_a=1-l$). Notice that $k_{(J,a)}=1-\deg e_a >0$ and so the unique fixed point $0\in\hat N$ of $E_\text{lift}$ is hyperbolic.

Let us define the vector field $\hat E$ on  $\hat N$ by
$$\hat E=E_\text{lift}+C_Q(B+i_E\,\cdot),$$
where  $$C_Q(B+i_E\,\cdot):({\textstyle\bigwedge}\, T^*N\otimes W)^1\to ({\textstyle\bigwedge}\, T^*N\otimes W)^1$$ is understood as a vertical vector field. 
$\hat E$ has the needed properties. Adding the vertical vector field $C_Q(B+i_E\,\cdot)$ to $E_\text{lift}$ moves the fixed point to a point $P\in\hat N$ lying over $0\in N$ and it is the unique fixed point of $\hat E$ (using the fact that the added vertical part is constant over $0\in N$, as $E=0$ at $0\in N$).  To see the hyperbolicity of the fixed point $P$ notice that the stable subspace of $T_P\hat N$ is the vertical subspace, with eigenvalues $-k_{(J,a)}$, and the unstable subspace  projects bijectively onto $T_0N$, with all eigenvalues equal to  $1$.
\end{proof}

Let us suppose that $N$ is compact.  For $r\geq1$ let $\Omega_{r}(N)$ denote the Banach space of $C^r$-forms. The previous two results (Theorem \ref{thm:extiso} and Proposition \ref{prop:K-inv}) remain valid when $A$, $B$ and $a$ are $C^r$-forms.

\begin{prop}\label{prop:kappa-emb}
The map
$$\kappa:\Omega_{r}(N,W|_U)^0\to\Omega_{r}(N,W)^0$$
\begin{equation}\label{eq:kappadef}
\kappa(A)=A-h(C_Q(A))
\end{equation}
is a smooth open embedding of Banach manifolds.
\end{prop}
\begin{proof}
The map $\kappa$ is smooth. By Proposition \ref{prop:K-inv} it is injective and its image is open. The inverse map (defined on the image of $\kappa$) was constructed in the proof of Proposition \ref{prop:K-inv} via the unstable manifold of the vector field $\hat E$, which depends on $B$. Since this dependence on $B$ is smooth, by differentiable dependence of unstable manifold on parameters (Robbin \cite[Theorem 4.1]{Robbin}) we get that $\kappa^{-1}$ is smooth. This implies that $\kappa$ is a smooth open embedding.
\end{proof}

A map similar to (\ref{eq:kappadef}) appears in the work of Kuranishi \cite{Kuranishi} and we will refer to it as the Kuranishi map.

\begin{thm}\label{thm:submf-loc}
The subset
$$\{A\in\Omega_r(N,W|_U)^0;\,dA-C_Q(A)=0\}=:\Omega_r(N,W|_U)^{MC}\subset \Omega_r(N,W|_U)^0,$$
is a (smooth) Banach submanifold and it is closed in the $C^0$-topology (i.e.\ in the $\sup$-norm). The subset
$$\Omega(N,W|_U)^{MC}\subset \Omega(N,W|_U)^0$$
is a Fr\'echet submanifold closed in the $C^0$-topology.
\end{thm}
\begin{proof}
By Theorem \ref{thm:extiso} we have 
$$\Omega_r(N,W|_U)^{MC}=\kappa^{-1}\bigl(\Omega_r(N,W)^{0,cl}\bigr).$$
Since $\kappa$ is an open embedding, the theorem follows from the $C^0$-continuity of $\kappa$ and from the fact that the space of closed forms 
$$\Omega_r(N,W)^{0,cl}\subset \Omega_r(N,W)^0$$
is a $C^0$-closed subspace (indeed $\alpha\in\Omega_r(N,W)^0$ is in the subspace $\Omega_r(N,W)^{0,cl}$ iff $\int_N\alpha\wedge d\beta=0$ for every  smooth form $\beta$  with support in the interior of $N$, which shows that $\Omega_r(N,W)^{0,cl}$ is $C^0$-closed). Since the result is true for every $r\geq1$, it also holds for $C^\infty$-forms.
\end{proof}

The rest of this section is a preparation for the gauge-fixing procedure of Section \ref{sec:gauge}.

It is somewhat inconvenient that $\Omega_r(N)$ is not a complex (i.e.\ that $d$ is not an everywhere-defined operator $\Omega_r(N)\to\Omega_r(N)$). Following A.~Henriques \cite{Henriques} let us consider the complex
$$\Omega_{r+}(N):=\{\alpha\in\Omega_r(N);d\alpha\in\Omega_r(N)\}$$
which is a Banach space with the norm 
$$\Vert\alpha\Vert_{r+}:=\Vert\alpha\Vert_{C^r}+\Vert d\alpha\Vert_{C^r}.$$
We can identify $\Omega_{r+}(N)$ with the closed subspace
$$\Gamma:=\{(\alpha,\beta);d\alpha=\beta\}\subset\Omega_r(N)\oplus\Omega_r(N),$$
i.e.\ with the graph of the unbounded operator $d:\Omega_r(N)\to\Omega_r(N)$. The isomorphism $\Gamma\to\Omega_{r+}(N)$ is given by the projection $(\alpha,\beta)\mapsto\alpha$.  

We have 
$$\Omega_r(N,W|_U)^{MC}\subset\Omega_{r+}(N,W)^0$$
 as for $A\in\Omega_r(N,W|_U)^{MC}$ we have $dA=C_Q(A)\in\Omega_{r}(N,W)$.
 
\begin{prop}
$\Omega_r(N,W|_U)^{MC}\subset\Omega_{r+}(N,W)^0$ is a Banach submanifold.
\end{prop}
\begin{proof}
Let us consider the embedding
$$e:\Omega_r(N,W)^0\to\Omega_r(N,W)^0\oplus\Omega_r(N,W)^1,\quad e(A)=(A,C_Q(A)).$$
For $A\in\Omega_r(N,W|_U)^{MC}$ we have $e(A)=(A,dA)$, hence $e$ embeds $\Omega_r(N,W|_U)^{MC}$ to $\Gamma\otimes W$. The isomorphism $\Gamma\otimes W\cong\Omega_{r+}(N,W)$ then implies that $\Omega_r(N,W|_U)^{MC}\subset\Omega_{r+}(N,W)^0$ is a submanifold.
\end{proof}

If $A_c\in\Omega(N,W|_U)^0$ is a constant (i.e.\ if the 0-form component of $A_c$ is a constant map $N\to U$ and the higher-form components of $A_c$ vanish) then $A_c\in\Omega(N,W|_U)^{MC}$ and $\kappa(A_c)=A_c$. We shall identify constant $A_c$'s with elements of $U$, i.e.\ we have an inclusion $U\subset\Omega(N,W|_U)^{MC}$ and $\kappa|_U=\on{id}_U$.

Let us notice that the map
$$A\mapsto A-hdA$$
is a projection
$$\Omega_{r+}(N,W)^0\to\Omega_{r+}(N,W)^{0,cl}=\Omega_{r}(N,W)^{0,cl}$$
which coincides with $\kappa$ on $\Omega_r(N,W|_U)^{MC}$. Let us now consider more general projections (equivalently, let us replace $h$ with another homotopy operator).

\begin{prop}\label{prop:locdiff}
Let $C\subset \Omega_{r+}(N)$ be a graded Banach subspace  such that
$\Omega_{r+}(N)=\Omega_r(N)^{cl} \oplus C$. Then there is an neighbourhood $U\subset\mathcal U\subset\Omega_r(N,W|_U)^{MC}$ such that the projection w.r.t.\ $(C\otimes W)^0$
$$\pi:\Omega_{r+}(N,W)^0\to \Omega_r(N,W)^{0,cl}$$
restricts to an open embedding $\mathcal U\to\Omega_r(N,W)^{0,cl}$ and to the indentity on $U$.
\end{prop}

\begin{proof}
Let $\pi^\text{res}:=\pi|_{\Omega_{r}(N,W|_U)^{MC}}$.
To show that there exists $\mathcal{U}$ with the desired property it is enough to show that the tangent map $T_{A_c}\pi^\text{res}$ is a linear isomorphism for each $A_c\in U$. We  know that the map 
$$\kappa^\text{res}:=\kappa|_{\Omega_{r}(N,W|_U)^{MC}}:\Omega_{r}(N,W|_U)^{MC}\to\Omega_{r}(N,W)^{0,cl}$$
 is an open embedding, and thus the tangent map $T_{A_c}\kappa^\text{res}$ is a linear isomorphism for every $A_c\in U$.
Explicitly,
$$T_{A_c}\kappa^\text{res}(A)=A-hC_\text{Q,\text{lin}}(A),$$
where
$$C_\text{Q,lin}^i(A)=\frac{\partial C^i_Q}{\partial \xi^j}(A_c)A^j$$
is the linearization of $C_Q$ at $A_c$.

Let us introduce a filtration $\mathcal{F}$ of the space $\Omega_{r+}(N,W)^0$ with
$$\mathcal{F}^i=\bigoplus_{k\leq i}\Omega_{r+}^k(N,W^{-k}).$$
To show that $T_{A_c}\pi^\text{res}$ is a linear isomorphism it is enough to verify that the  linear map
$T_{A_c}\pi^\text{res}\circ(T_{A_c}\kappa^\text{res})^{-1}$ is filtered and that it
induces the identity map on the associated graded, and thus is a linear isomorphism. Let $B\in \Omega_{r}(N,W)^{0,cl}$ and $B\in\mathcal{F}^i$. Then
$$(T_{A_c}\kappa^\text{res})^{-1}(B)=B+hC_{Q,\text{lin}}\bigl((T_{A_c}\kappa^\text{res})^{-1}(B)\bigr).$$
Since $hC_{Q,\text{lin}}\bigl((T_{A_c}\kappa^\text{res})^{-1}(B)\bigr)\in\mathcal{F}^{i-1}$ and $B$ is closed, we see that
$$T_{A_c}\pi^\text{res}\circ(T_{A_c}\kappa^\text{res})^{-1}(B)=B\mod\mathcal{F}^{i-1}$$
as we wanted to show.
\end{proof}

\section{Filling horns}\label{sec:fh}

Let $\Delta^n$ be the $n$-dimensional simplex. By Theorem \ref{thm:submf-loc}, the sets
$$K_n^{big}(\mathcal{A}_{W|_U},Q):=\Omega(\Delta^n,W|_U)^{MC}$$
are naturally Fr\'echet manifolds. The affine maps between simplices then make the collection of $K_n^{big}$'s to a simplicial Fr\'echet manifold.

For any simplicial set $S_\bullet$ and any $0\leq k\leq n$ let $S_{n,k}$ denote the corresponding horn.\footnote{Let
us recall the definitions: a simplicial set is a contravariant functor from the category $\Delta$, whose objects are the finite sets $[n]=\{0,1,\dots,n\}$ ($n=0,1,2,\dots$) and morphisms the non-decreasing maps, to the category of sets. If $S$ is such a functor, we use the notation $S_n:=S([n])$. The simplest examples are the simplicial sets $\Delta[n]_m:=\on{Hom}_\Delta([m],[n])$, which satisfy $S_n=\on{Hom}(\Delta[n],S)$ for any simpicial set $S$. Another example is given by the ``horn'' simplicial set $\Lambda[n,k]$ ($0\leq k\leq n$) given by $\Lambda[n,k]_m:=\bigl\{f\in\on{Hom}_\Delta([m],[n]);\bigl| f([m])\setminus\{k\}\bigr|\leq n-2\bigr\}$. If $S$ is a simplicial set and if $0\leq k\leq n$ then the corresponding horn is the set $S_{n,k}:=\Hom(\Lambda[n,k],S)$. The morphism $\Lambda[n,k]\to\Delta[n]$, given by the inclusion, gives us a map $S_n\to S_{n,k}$.
}
 The $k$-th horn of the geometric $n$-simplex $\Delta^n$ (i.e.\ the union of the $n-1$-dimensional faces containing the $k$-th vertex) will be denoted $\Lambda^n_k$.

\begin{defn}[A.~Henriques]\label{def:Kanm}
A simplicial manifold $K_\bullet$ is \emph{Kan} if the map $K_n\to K_{n,k}$ is a surjective submersion for any $0\leq k\leq n$.
\end{defn}

Here a smooth map $f:X\to Y$ between two Fr\'echet manifolds is called a \emph{submersion} if locally, up to local diffeomorphisms, $f$ is a surjective continuous linear map between Fr\'echet spaces, admitting a right inverse.
To make Definition \ref{def:Kanm} meaningful, one needs to  check (inductively in $n$) that $K_{n,k}$ are manifolds ($K_{n,k}$ is defined as the horn of $K_\bullet$ seen as a simplicial set). This was done by Henriques \cite{Henriques}.

Any simplicial topological vector space $X_\bullet$ is automatically a Kan simplicial manifold. Indeed, there is an explicit continuous linear map $X_{n,k}\to X_n$ due to Moore \cite{Moore} (in the proof of his theorem stating that any simplicial group is Kan) which is right-inverse to the horn map $X_n\to X_{n,k}$. Namely, supposing without loss of generality that $k=n$, if $(y_0,\dots,y_{n-1})\in X_{n,n}$ ($y_i\in X_{n-1}$), one defines $w_0,\dots, w_{n-1}\in X_n$ via 
\begin{equation}\label{eq:moore}
w_0=s_{0}y_0,\quad w_i=w_{i-1}-s_id_iw_{i-1}+s_iy_i
\end{equation}
 ($i=1,\dots,n-1$, $d_i$ and $s_i$ are the face and degeneracy maps respectively), and then $d_iw_{n-1}=y_i$ for all $0\leq i\leq n-1$, i.e.\ $w_{n-1}$ fills the horn $(y_0,\dots,y_{n-1})\in X_{n,n}$.

Here is the principal result of this section (it is valid also for $C^r$-forms, when we get Kan simplicial Banach manifolds).
\begin{thm}\label{thm:kbig-loc}
$K^{big}_\bullet(\mathcal{A}_{W|_U},Q)$ is a Kan simplicial Fr\'echet manifold.
\end{thm}
\begin{proof}

Let $K_\bullet:=K^{big}_\bullet(\mathcal{A}_{W|_U},Q)$ and let $X_\bullet:=\Omega(\Delta^\bullet,W)^{0,cl}$. As $X_\bullet$ is a simplicial Fr\'echet vector space, it is a Kan simplicial manifold.

Let us first prove that the horn maps $K_n\to K_{n,k}$ are submersion. Let us place $\Delta^n$ to $\R^n$ so that the $k$'th vertex is at $0\in\R^n$. By applying the Kuranishi map $\kappa$ we get the commutative square
$$
\begin{tikzcd}
K_{n}\arrow{r}{\kappa}\arrow{d} & X_n\arrow{d}\\
K_{n,k}\arrow{r}{\kappa} & X_{n,k}
\end{tikzcd}
$$
As the horizontal arrows are open embeddings and the vertical arrow $X_{n}\to X_{n,k}$ is a submersion (in fact a continuous linear map admitting a right inverse), $K_n\to K_{n,k}$ is also a submersion.

It remains to prove that $K_n\to K_{n,k}$ is surjective. 
Let   $A_\text{horn}\in K_{n,k}$, let $B_\text{horn}:=\kappa(A)\in X_{n,k}$. We choose $B\in X_n$ extending $B_\text{horn}$.

By Proposition \ref{prop:K-inv}, Equation \eqref{eq:a} has a solution $a$ on an open subset  $N'\subset\Delta^n$ such that $\Lambda^n_k\subset N'$. The form
$$A'=B+i_Ea$$
thus satisfies
$$A'\in \Omega\bigl(N', W|_{U}\bigr)^{MC},\  
A'|_{\Lambda^n_k}=A_\text{horn}.$$

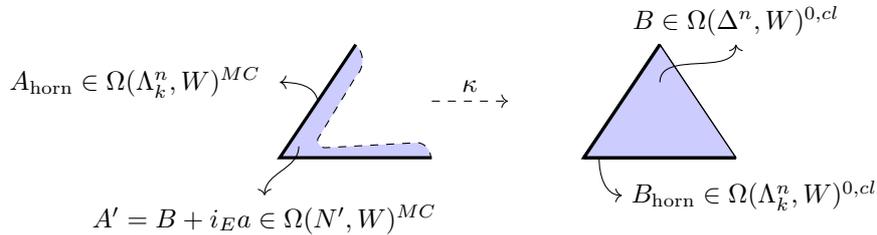
\begin{figure}[h!tb]
\center
\begin{tikzpicture}
\draw[->]  (0.5,0.75) to[out=100, in=0] node[anchor=west,left,pos=1.2]{$A_{\text{horn}}\in\Omega(\Lambda^n_k,W)^{MC}$} (0,1);
\draw[->,dashed] (2,0.75) -- (3,0.75)  node[anchor=north,above,pos=0.5]{$\kappa$};
\draw[very thick] (4,0) -- (6,0);
\draw[very thick] (4,0) -- (5,1.5);
\draw[->] (4.2,0) to[out=-90,in=180] (4.5,-0.5) node[anchor=east,right]{$B_\text{horn}\in\Omega(\Lambda^n_k,W)^{0,cl}$};
\filldraw[fill=blue!20] 
(4,0) -- (6,0) -- (5,1.5) -- (4,0);
%\draw[very thick] (4,0) -- (6,0);
\draw[very thick] (6,0)--(4,0) -- (5,1.5);
\draw (5,1.5) -- (6,0);
\draw[->] (5,1) to[out=90,in=-90] (6,1.5) node[anchor=north,above]{$B\in\Omega(\Delta^n,W)^{0,cl}$};
\filldraw[fill=blue!20,dashed]
(0,0) -- (1,1.5) arc (45:-45:3mm) -- (0.5,0.25) arc (-190:-90:1mm) -- (1.82,0.2) arc(80:0:2mm);
\draw[very thick] (2,0)--(0,0)--(1,1.5);
\draw[->] (0.25,0.15) to[out=-90,in=90] (-0.2,-0.5) node[anchor=south,below]{$A'=B+i_Ea\in\Omega(N',W)^{MC}$}; 
\end{tikzpicture}
\caption{Filling horns in $K^\text{big}_\bullet$.}
\label{fig:hornfill}
\end{figure}

Let now $\phi\in C^\infty(\Delta^n)$ satisfy $0\leq\phi\leq1$, $\phi|_{\Lambda^n_k}=1$, and $\phi|_{\Delta^n\setminus N'}=0$.
Let $h:\Delta^n\to N'$ be given by $x\mapsto \phi(x)x$. Then
$$A:=h^*A'$$
satisfies
$$A\in \Omega\bigl(\Delta^n, W|_{U}\bigr)^{MC}=K_n,\  
A|_{\Lambda^n_k}=A_\text{horn}.$$
This proves surjectivity of $K_n\to K_{n,k}$, as the map sends $A\in K_n$ to $A_\text{horn}\in K_{n,k}$.
\end{proof}

\section{Globalization}\label{sec:glob}
So far we considered dg algebras of the form $\mathcal A_{W|_U}$. Let now $V\to M$ be a negatively-graded vector bundle, and let us consider the algebra $\mathcal A_V$ with a chosen differential $Q$. In this section we shall prove that Theorems \ref{thm:submf-loc} and \ref{thm:kbig-loc} hold also in this more general setting.

To prove these results we need to consider the following relative situation. Let $N\subset\R^n$ be a compact star-shaped $n$-dimensional submanifold with corners. The restriction $\mathcal A_{T[1]\R^n}=\Omega(\R^n)\to\Omega(N)$ is a morphism of dgcas. The corresponding element $A_0\in\Omega(N,T[1]\R^n)^{MC}$ is 
$$A_0=\sum_k E_k x^k +e_k dx^k$$
where $x^k$ are the coordinates on $N\subset\R^n$ and $E_k$ and $e_k$ is the standard basis of $\R^n$ and of $\R^n[1]$ respectively. Application of $\kappa$ gives
$$B_0:=\kappa(A_0)=\sum_ke_k dx^k\in\Omega(N,T[1]\R^n)^{0,cl}.$$

Suppose now that $W$ is a non-positively graded vector space and $p:W\to T[1]\R^n$ a surjective graded linear map. Let $U\subset W^0$ be open, $Q$ be a differential on $\mathcal A_{W|_U}$, and let us suppose that the pullback $p^*:\mathcal A_{T[1]\R^n}=\Omega(\R^n)\to\mathcal A_{W|_U}$ is a morphism of dgcas (i.e.\ that $p:W|_U\to T[1]\R^n$ is a morphism of NQ-manifolds). Let
$$\Omega(N,W|_U)^{0,A_0}:=\{A\in\Omega(N,W|_U)^{0}; p(A)=A_0\}$$
and similarly
$$\Omega(N,W)^{0,B_0}:=\{B\in\Omega(N,W)^{0}; p(B)=B_0\}.$$

The elements of
$$\Omega(N,W|_U)^{MC,A_0}:=\Omega(N,W|_U)^\textit{MC}\cap\Omega(N,W|_U)^{0,A_0}$$
correspond to those dgca morphisms $\mathcal A_{W|_U}\to\Omega(N)$ for which the diagram
$$
\begin{tikzcd}
\Omega(N)& \mathcal A_{W|_U}\arrow{l}\\
&\Omega(\R^n)\arrow{u}[swap]{p^*}\arrow{ul}
\end{tikzcd}
$$
is commutative, i.e.\ of those morphisms $T[1]N\rsa W|_U$ of NQ-manifolds for which
$$
\begin{tikzcd}
T[1]N\arrow[squiggly]{r}\arrow[squiggly,swap]{dr} & W|_U\arrow[squiggly]{d}{p}\\
&T[1]\R^n
\end{tikzcd}
$$
is commutative.

\begin{prop}\label{prop:rel-mfld}
$$\Omega(N,W|_U)^{MC,A_0}=\kappa^{-1}\bigl(\Omega(N,W)^{0,cl,B_0}\bigr).$$
 In particular, $\Omega(N,W|_U)^{MC,A_0}\subset\Omega(N,W|_U)^{0,A_0}$ is a Fr\'echet submanifold closed in the $C^0$ topology.
\end{prop}
\begin{proof}
The equality follows from Theorem \ref{thm:extiso} and from $B_0=\kappa(A_0)$. The fact that 
$$\Omega(N,W|_U)^{MC,A_0}\subset\Omega(N,W|_U)^{0,A_0}$$
 is a Fr\'echet submanifold closed in the $C^0$ topology then follows from the fact that $\kappa$ is an open embedding and $C^0$-continuous, and from the fact that $\Omega(N,W)^{0,cl,B_0}$ is a $C^0$-closed affine subspace of $\Omega(N,W)^{0}$.
\end{proof}

We can now prove our first globalized result.

\begin{thm}\label{thm:submfd}
 Let $N\subset\R^n$ be a star-shaped compact $n$-dimensional submanifold with corners, $V\to M$ a negatively graded vector bundle, and $Q$ a differential on the gca $\mathcal A_V$. Then the subset
$$\{\text{dg-morphisms }\mathcal A_V\to\Omega(N)\}\subset\{\text{graded algebra morphisms }\mathcal A_V\to\Omega(N)\},$$
i.e.
$$\Omega(N,V)^{MC}\subset \Omega(N,V)^0,$$
is a smooth Fr\'echet submanifold  closed in the $C^0$-topology.
\end{thm}
\begin{proof}
For any smooth map $f:N\to M$ we can find an open subset $U'\subset\R^n\times M$ containing  the graph of $f$ such that
\begin{enumerate}
\item there is an open subset $U\subset\R^n\times\R^m$ ($m=\dim M$) and a diffeomorphism $g:U'\to U$ such that $g$ is trivial over $\R^n$ (i.e.\ such that $\forall (x,y)\in U'\subset\R^n\times M$, $g(x,y)=(x,\tilde y(x,y))$ for some $\tilde y(x,y)\in\R^m$)
\item the graded vector bundle $(p_M^* V)|_{U'}$ is trivial (where $p_M:\R^n\times M\to M$ is the projection).
\end{enumerate}

As a result, there is a non-positively graded vector space $\tilde W$ with $\tilde W^0=\R^m$ (a local model of $V\to M$), an open subset $U\subset W^0$ where $W=T[1]\R^n\oplus\tilde W$ (so that $W^0=\R^{n+m}$) and an isomorphism of graded vector bundles
$$ t:(T[1]\R^n\times V)|_{U'}\to W|_U$$
such that the triangle
\begin{equation}\label{eq:aux-tri}
\begin{tikzcd}
(T[1]\R^n\times V)|_{U'}\arrow{rr}{t} \arrow{rd}[swap]{p'}& & W|_U\arrow{ld}{p}\\
& T[1]\R^n &
\end{tikzcd}
\end{equation}
commutes, where $p$ and $p'$  are the projections. Under the resulting isomorphism of gcas
$$t^*:\mathcal A_{W|_U}\cong \mathcal A_{(T[1]\R^n\times V)|_{U'}}$$
the differential $d+Q$ on   $A_{(T[1]\R^n\times V)|_{U'}}$ is sent to a differential $\tilde Q$ on $\mathcal A_{W|_U}$ such that 
$$p^*:\Omega(\R^n)\to \mathcal A_{W|_U}$$
is a dg morphism, as follows from the commutativity of  \eqref{eq:aux-tri}.

By Proposition \ref{prop:rel-mfld} we know that $\Omega(N,W|_U)^{MC,A_0}\subset\Omega(N,W|_U)^{0,A_0}$ is a $C^0$-closed submanifold. Let $\Omega(N,V)^0|_{U'}\subset\Omega(N,V)^0$ be the subset of those elements for which the graphs of their function parts $N\to M$ lie in $U'$. By construction $t^*$ restricts to a diffeomorphism $\Omega(N,W|_U)^{0,A_0}\cong \Omega(N,V)^0|_{U'}$ and  the image of $\Omega(N,W|_U)^{MC,A_0}$ is $\Omega(N,V)^{MC}\cap\Omega(N,V)^0|_{U'}$. As a result $\Omega(N,V)^{MC}\cap\Omega(N,V)^0|_{U'}\subset \Omega(N,V)^0|_{U'}$ is a $C^0$-closed submanifold, and thus also $\Omega(N,V)^{MC}\subset \Omega(N,V)^0$ is a $C^0$-closed submanifold, as the subsets $\Omega(N,V)^0|_{U'}\subset\Omega(N,V)^0$ form a $C^0$-open cover of $\Omega(N,V)^0$.
\end{proof}

\begin{rem}
Another proof of Theorem \ref{thm:submfd}, avoiding Proposition \ref{prop:rel-mfld}, is as follows. If  $\Omega(N,V)^{MC}|_{U'}$ is non-empty, one can prove that (after possibly decreasing $U'$) there is a NQ-manifold $\tilde W|_{\tilde U}$ and an open embedding of NQ-manifolds 
$$(T[1]\R^n\times V)|_{U'}\rsa T[1]\R^n\times \tilde W|_{\tilde U}$$
commuting with the projections to $T[1]\R^n$. As a result we can identify $\Omega(N,V)^{MC}|_{U'}\subset \Omega(N,V)^0|_{U'}$ with $\Omega(N,\tilde W|_{\tilde U})^{MC}\subset\Omega(N,\tilde W|_{\tilde U})^{0}$, which is a submanifold by Theorem \ref {thm:submf-loc}.
\end{rem}

Our second globalized result will be proved by similar methods.

\begin{thm}\label{ref:kbig}
The Fr\'echet simplicial manifold $\Omega(\Delta^\bullet,V)^{MC}$ is Kan.
\end{thm}
\begin{proof}
Let us use the notation $K_\bullet:=\Omega(\Delta^\bullet,V)^{MC}$.
Let us first prove that the map $K_n\to K_{n,k}$ is surjective. Any element of $K_{n,k}$ gives us, in particular, a map $f_\text{horn}:\Lambda^n_k\to M$. Let us extend $f_\text{horn}$ to a map $f:\Delta^n\to M$. As in the proof of Theorem \ref{thm:submfd}, let $U'\subset\R^n\times M$ be an open subset such that $U'$ contains the graph of $f$, and such that
\begin{enumerate}
\item $U'$ is diffeomorphic over $\R^n$ to an open subset $U'$ of $\R^n\times\R^m$ ($m=\dim M$)
\item the graded vector bundle $p_M^*V|_{U'}\to U'$ is trivial.
\end{enumerate}
This ensures the existence of an isomorphism 
$$t:(T[1]\R^n\times V)|_{U'}\cong(T[1]\R^n\times \tilde W)|_{U}$$
for some non-positively graded vector space $\tilde W$ with $\tilde W^0=\R^m$ and some open subset $U\subset W^0$, where $W:=T[1]\R^n\oplus\tilde W$.

Let us choose an element  $A_\text{horn}\in K_{n,k}$, i.e.\footnote{By a differential form on the horn $\Lambda^n_k$ we mean a differential form on each of its faces which agree on the overlaps.}
$$A_\text{horn}\in\Omega\bigl(\Lambda^n_k,W|_U)^{MC,A_0}.$$
After we apply the map $\kappa$ to $A_\text{horn}$, we get a closed form
$$B_\text{horn}\in\Omega(\Lambda^n_k,W)^{0,cl}.$$
and extend $B_\text{horn}$ to a closed form
$$B\in\Omega\bigl(\Delta^n,W\bigr)^{0,cl}.$$

By Proposition \ref{prop:K-inv}, Equation \eqref{eq:a} has a solution $a$ on an open subset  $N'\subset\Delta^n$ such that $\Lambda^n_k\subset N'$. The form
$$A'=B+i_Ea$$
thus satisfies
$$A'\in \Omega\bigl(N',W|_{U}\bigr)^{MC},\  
A'|_{\Lambda^n_k}=A_\text{horn}.$$

Let now $\phi\in C^\infty(\Delta^n)$ satisfy $0\leq\phi\leq1$, $\phi|_{\Lambda^n_k}=1$, and $\phi|_{\Delta^n\setminus N'}=0$.
Let $h:\Delta^n\to N'$ be given by $x\mapsto \phi(x)x$. Then
$$A:=h^*A'$$
satisfies
$$A\in \Omega\bigl(\Delta^n,W|_{U}\bigr)^{MC},\  
A|_{\Lambda^n_k}=A_\text{horn}.$$
The isomorphism $t$ then allows us to see $A$ as an element of $\Omega\bigl(\Delta^n,(T[1]\R^n\times V)|_{U'}\bigr)^{MC}$. We  project it to get an element of $\Omega(\Delta^n,V)^{MC}=K_n$, which finally shows that the horn map $K_n\to K_{n,k}$ is surjective.

Let us now prove that $K_n\to K_{n,k}$ is a submersion. For an open $U'\subset\R^n\times M$ satisfying the conditions (1) and (2) above, let $K'_n:=\Omega(\Delta^n,V)^{MC}|_{U'}\subset\Omega(\Delta^n,V)^{MC}=K_n$ be the subset of those elements for which the graph of their function part lies in $U'$. The isomorphism of graded vector bundles gives us a diffeomorphism
$$\on{id}_{\Omega(\Delta^n)}\otimes t:K_n'=\Omega(\Delta^n,V)^{MC}|_{U'}\cong\Omega(\Delta^n,W|_U)^{MC,A_0}$$
and so we have a commutative square
$$
\begin{tikzcd}
K'_{n}\arrow{rr}{\kappa\circ(\on{id}_{\Omega(\Delta^n)}\otimes t)}\arrow{d} && \Omega(\Delta^n,W)^{0,cl,B_0}\arrow{d}\\
K'_{n,k}\arrow{rr}{\kappa\circ(\on{id}_{\Omega(\Delta^n)}\otimes t)} && \Omega(\Lambda^n_k,W)^{0,cl,B_0}
\end{tikzcd}
$$
where the horizontal arrows are open embeddings and the right vertical arrow is a submersion, hence also the left vertical arrow is a submersion. Thus $K_n\to K_{n,k}$ is a submersion.
\end{proof}

\section{Gauge fixing}\label{sec:gauge}
The simplicial manifold $K^\text{big}_\bullet$ is infinite-dimensional. We follow Getzler's \cite{Getzler} approach and define a finite-dimensional simplicial submanifold $K^s_\bullet\subset K^\text{big}_\bullet$. In the case of a Lie algebra (or algebroid) $K^s_\bullet$ is the nerve of the corresponding local Lie group (or groupoid). In general $K^s_\bullet$ is a local Lie $\ell$-groupoid (Definition \ref{def:l-groupoid}). $K^s_\bullet$ depends on a choice of a gauge condition; different gauges lead to isomorphic (by non-canonical isomorphisms) local (weak) Lie $\ell$-groupoids. This construction is local in nature. In particular, we will not need the results of Sections  \ref{sec:fh} and \ref{sec:glob}.

Let us fix an integer $r\geq 1$.
Consider the cochain complex of elementary Whitney forms in $\Omega_{r+}(\Delta^n)$ on the geometric $n-$simplex 
$$E(\Delta^n)=\bigoplus_{k\geq 0}E^k(\Delta^n)\subset\Omega_{r+}(\Delta^n).$$ 
It is given in each degree by the linear span of differential forms defined using standard coordinates $(t_0,\ldots,t_n, \sum_k t_k=1)$ on $\Delta^n$ by the formulas
$$\omega_{i_0\ldots i_k}=k!\sum_{j=0}^k(-1)^jt_{i_j}dt_{i_0}\ldots \widehat{dt_{i_j}}\ldots dt_{i_k}$$
where $\widehat{dt_{i_j}}$ means the omission of $dt_{i_j}$. The $k!$ factor is prescribed so that the integral over the $k-$dimensional sub-simplex given by the sequence of $k+1$ vertices $(e_{i_0},\ldots,e_{i_k})$ of $\Delta^n$ is
$$\int_{(e_{i_0},\ldots,e_{i_k})}\omega_{i_0\ldots i_k}=1$$
(the integral of $\omega_{i_0\ldots i_k}$ over any other sub-simplex vanishes).

$E(\Delta^n)$ is a sub-complex of $\Omega_{r+}(\Delta^n)$, as
$$d\,{\omega_{i_0\ldots i_k}}=\sum_{i=0}^n\omega_{ii_0\ldots i_k}.$$
There is an explicit projection $p_\bullet$ given by Whitney \cite{Whitney}
$$p_n:\Omega_{r+}(\Delta^n)\to E(\Delta^n)$$
$$p_n\alpha=\sum_{k=0}^n\sum_{i_0<\cdots<i_k}\omega_{i_0\ldots i_k}\int_{(e_{i_0},\ldots,e_{i_k})}\alpha$$
compatible with the simplicial (cochain complex) structures of $\Omega_{r+}(\Delta^\bullet)$ and $E(\Delta^\bullet)$. $E(\Delta^\bullet)$ is furthermore isomorphic (with the isomorphism respecting the simplicial structures) to the complex of simplicial cochains on $\Delta^\bullet$.

Let us recall that a \emph{symmetric simplicial object} is a contravariant functor from the category of the finite sets of the form $\{0,1,\dots,n\}$ with all maps (i.e.\ not only non-decreasing maps) as morphisms. Observe that $K^\text{big}_\bullet$, $\Omega_{r+}(\Delta^\bullet)$, $E(\Delta^\bullet)$ are naturally symmetric simplicial objects and that $p_\bullet$ is a symmetric simplicial map.
Following Getzler \cite{Getzler} let us make the following Definition. 
\begin{defn}
A \emph{gauge} on $\Omega_{r+}(\Delta^\bullet)$  is a  continuous symmetric simplicial linear map 
$$s_\bullet:\Omega_{r+}(\Delta^\bullet)\to \Omega_{r+}(\Delta^\bullet)$$
of degree $-1$ satisfying
$$id_{\Omega_{r+}(\Delta^\bullet)}-p_\bullet=[d,s_\bullet]$$
and
$${s_\bullet}^2=0,\ p_\bullet s_\bullet=s_\bullet p_\bullet=0.$$
\end{defn} 
  Restricting to $\ker s_\bullet$ can be thought of as gauge fixing. An important example of a  a gauge is given by Dupont \cite{Dupont} in the proof of the simplicial de Rham theorem.  Its construction is as follows. Let us denote $h_{i}$ the de Rham homotopy operator associated to the retraction of $\Delta^\bullet$ to the $i$-th vertex. Then the operators
$$s_n=\sum_{k=0}^{n-1}\sum_{i_0<\ldots< i_k}\omega_{i_0\ldots i_k}h_{i_k}\ldots h_{i_0}$$
where $n\geq 0$ form a gauge $s_\bullet$ (it is $\Vert\cdot\Vert_{r+}$-continuous for every $r$).

\begin{notation}
We shall set 
$$\Omega_{r+}(\Delta^\bullet)^s:=\ker s_\bullet\subset \Omega_{r+}(\Delta^\bullet),$$
 $$\Omega_{r+}(\Delta^\bullet)^p:=\ker p_\bullet\subset \Omega_{r+}(\Delta^\bullet),$$
 $$\Omega_{r+}(\Delta^\bullet)^{s,p}:=\ker s_\bullet\cap\ker p_\bullet\subset \Omega_{r+}(\Delta^\bullet).$$
More generally, if $W$ is a graded vector space, we let
$$\Omega_{r+}(\Delta^\bullet,W)^s:=\ker (s_\bullet\otimes\on{id}_W)\subset \Omega_{r+}(\Delta^\bullet,W)$$
etc.
\end{notation}

The meaning of a gauge is summed-up by the following proposition.
\begin{prop}\label{prop:gauge}
A gauge is equivalent to a choice of a closed symmetric simplicial subspace
$$\mathscr M(\Delta^\bullet)\subset\Omega_{r+}(\Delta^\bullet)^p$$
such that
$$d:\mathscr M(\Delta^\bullet)\to\Omega_r(\Delta^\bullet)^{p,cl}$$
is a bijection, i.e.\ such that $\Omega_{r+}(\Delta^\bullet)^{p}=\Omega_r(\Delta^\bullet)^{p,cl}\oplus\mathscr M(\Delta^\bullet)$. If $s_\bullet$ is given then 
$$\mathscr M(\Delta^\bullet)= \Omega_{r+}(\Delta^\bullet)^{s,p}.$$ 
If $\mathscr M(\Delta^\bullet)$ is given then $s_\bullet$ is
\begin{equation}\label{slpgauge}
\begin{tikzpicture}[baseline]
\matrix(m)[matrix of math nodes,
row sep=3em, column sep=0.4em,
text height=1.5ex, text depth=0.25ex]
{0 & E(\Delta^\bullet) & \oplus & \mathscr M(\Delta^\bullet) & \oplus & \Omega_r(\Delta^\bullet)^{p,cl}. \\ };
\path[->,font=\scriptsize]
(m-1-2) edge node[above] {$s_\bullet$} (m-1-1)
(m-1-4) edge [bend left] node[auto] {$s_\bullet$} (m-1-1)
(m-1-6) edge [bend right] node[above] {$s_\bullet=d^{-1}$} (m-1-4);
\end{tikzpicture}
\end{equation}
\end{prop}

\begin{proof}
The map $s_\bullet$ defined  by \eqref{slpgauge} clearly satisfies 
$$s_\bullet p_\bullet=0,\ p_\bullet s_\bullet=0,\ \on{id}_{\Omega_{r+}(\Delta^\bullet)}-p_\bullet=[d,s_\bullet],\ s_\bullet^2=0$$
so it is a gauge, and $\mathscr M(\Delta^\bullet)=\Omega_{r+}(\Delta^\bullet)^{s,p}$.

Conversely given a gauge $s_\bullet$ we set $\mathscr M(\Delta^\bullet)=\Omega_{r+}(\Delta^\bullet)^{s,p}$ and we easily see that 
 $d:\mathscr M(\Delta^\bullet)\to\Omega_{r}(\Delta^\bullet)^{p,cl}$ is an isomorphism, and that the gauge defined by \eqref{slpgauge} coincides with the original $s_\bullet$.
\end{proof}

We can use $s_\bullet$ to specify the closed graded  subspace $C\subset\Omega_{r+}(\Delta^\bullet)$ of Proposition \ref{prop:locdiff}.
 Let 
 $$D(\Delta^\bullet):=h_0\bigl(E(\Delta^\bullet)\bigr)\subset E(\Delta^\bullet),$$
 (where $h_0$ is the de Rham homotopy operator given by the contraction to the vertex 0), so that  $E(\Delta^\bullet)=E(\Delta^\bullet)^{cl}\oplus D(\Delta^\bullet).$ 
We set
$C=D(\Delta^\bullet)\oplus\mathscr M(\Delta^\bullet)$
and consider the projection
$$\pi_\bullet:\Omega_{r+}(\Delta^\bullet)\to\Omega_r(\Delta^\bullet)^{cl}$$
w.r.t.\ $C$.

Notice that $D(\Delta^\bullet)\subset E(\Delta^\bullet)$ is not a simplicial subspace, but it is compatible with the maps between simplices that preserve the vertex $0$. Likewise, $\pi_\bullet$ is not a simplicial map, but it is (symmetric) 0-simplicial in the following sense:
\begin{defn}
If $X_\bullet$ and $Y_\bullet$ are simplicial sets then a sequence of maps $f_\bullet:X_\bullet\to Y_\bullet$ is a \emph{0-simplicial map} if it is functorial under the order-preserving maps $\{0,1,\dots,n\}\to\{0,1,\dots,m\}$ sending $0$ to $0$. If $X_\bullet$ and $Y_\bullet$ are symmetric simplicial then $f_\bullet$ is \emph{symmetric 0-simplicial} provided it is functorial wrt.\ all 0-preserving maps.
\end{defn}

As before let $W$ be a finite-dimensional non-positively graded vector space,  $U\subset W^0$ be an open subset, and let $Q$ be a differential on the algebra $\mathcal{A}_{W|_U}$ and $\kappa$ be the corresponding Kuranishi map given by (\ref{eq:kappadef}). Furthermore let $\mathcal{U}_\bullet$ be the open neighborhood $U\subset\mathcal{U}_\bullet\subset\Omega_r(\Delta^\bullet,W|_U)^{MC}$ from Proposition \ref{prop:locdiff}, where $ C=D(\Delta^\bullet)\oplus\mathscr M(\Delta^\bullet)$.  We can demand $\mathcal{U}_\bullet$ to be a symmetric simplicial submanifold.

 We will denote the symmetric simplicial set of gauge-fixed dg-morphisms in $\mathcal{U}_\bullet$ by
$$K^s_\bullet(\mathcal{A}_{W|_U},Q):=\Omega_r(\Delta^\bullet,W|_U)^{MC,s}:=\mathcal{U}_\bullet\cap \Omega_{r+}(\Delta^\bullet,W)^{0,s}.$$ 
Its elements are (sufficiently small) forms $A\in\Omega_{r}(\Delta^\bullet,W|_U)^0$ satisfying the equations
$$dA=C_Q(A),\quad s_\bullet A=0.$$

Let us now recall the definition of higher Lie groupoids. As is usual we actually define higher groupoids as nerves. 

\begin{defn}(Getzler \cite{Getzler}, Henriques \cite{Henriques}, Zhu \cite{Zhu})\label{def:l-groupoid}
A Kan simplicial manifold $K_\bullet$ is  a  \emph{Lie $\ell$-groupoid} if the maps $K_n\to K_{n,k}$ are diffeomorphisms for all $n>\ell$ and all $0\leq k\leq n$.
A  simplicial manifold $K_\bullet$ is  a  \emph{local Lie $\ell$-groupoid} if the maps $K_n\to K_{n,k}$ are submersions for all $n\geq 1$, $0\leq k\leq n$ and open embeddings for all $n>\ell$, $0\leq k\leq n$.
\end{defn}

\begin{thm}\label{thm:ksmall}
The simplicial set $K^s_\bullet(\mathcal{A}_{W|_U},Q)$ is a finite-dimensional  local Lie $\ell$-groupoid with $-\ell$ being the lowest degree of $W$.
\end{thm}
\begin{proof}
 Let
$$\pi_\bullet:\Omega_{r+}(\Delta^\bullet,W)^0\to\Omega_r(\Delta^\bullet,W)^{0,cl}$$
be the projection w.r.t.\ $( C\otimes W)^0$ (where $C=D(\Delta^\bullet)\oplus\mathscr M(\Delta^\bullet)$)
and let
$$\pi^{res}_\bullet:\mathcal U_\bullet\to \Omega_r(\Delta^\bullet,W)^{0,cl}$$
be the restriction of $\pi_\bullet$ to $\mathcal U_\bullet$.
  By Proposition \ref{prop:locdiff}, $\pi^{res}_\bullet$ is an open embedding. Let us recall that $\pi^{res}_\bullet$ is  a 0-simplicial map.

Let $K^s_\bullet:=K^s_\bullet(\mathcal{A}_{W|_U},Q)$.
We have
\begin{multline*}
K^s_\bullet=\mathcal{U}_\bullet\cap\Omega_{r+}(\Delta^\bullet,W)^{0,s}\\
=(\pi^{res}_\bullet)^{-1}\bigl( \Omega_r(\Delta^\bullet,W)^{0,cl}\cap\Omega_{r+}(\Delta^\bullet,W)^{0,s}\bigr)
=(\pi^{res}_\bullet)^{-1}\bigl(E(\Delta^\bullet, W)^{0,cl}\bigr)
\end{multline*}
as $\Omega_r(\Delta^\bullet,W)^{0,cl}\cap\Omega_{r+}(\Delta^\bullet,W)^{0,s}=E(\Delta^\bullet,W)^{0,cl}$.
This implies that each degree of the simplicial set $K^s_\bullet$ is a finite-dimensional smooth manifold and that 
\begin{equation}\label{eq:oembKs}
\pi^{res}_\bullet:K^{s}_\bullet\to E(\Delta^\bullet, W)^{0,cl}
\end{equation}
is an open embedding. 

It remains to prove that it is a local Lie $\ell$-groupoid. Since $E(\Delta^\bullet,W)^{0,cl}$ is a Lie $\ell$-groupoid and $\pi^{res}_\bullet$ a 0-simplicial map, $K^s_\bullet$ satisfies the required conditions for the projections $K^s_n\to K^s_{n,0}$. As $K^s_\bullet$ is symmetric simplicial, the same is true for all horn maps $K^s_n\to K^s_{n,k}$.
\end{proof}

The simplicial manifold $K^s_\bullet(\mathcal A_{W|_U},Q)=\Omega_r(\Delta^\bullet, W|_U)^{MC,s}$ was constructed using infinite-dimensional techniques. We can now see it as a simplicial submanifold of a finite-dimensional simplicial vector space:
\begin{thm}
The projection 
$$E(\Delta^\bullet, W)^0\oplus \mathscr M(\Delta^\bullet,W)^0\to E(\Delta^\bullet, W)^0$$
restricts to an embedding of symmetric simplicial manifolds
$$K^s_\bullet(\mathcal A_{W|_U},Q)\to E(\Delta^\bullet, W)^0.$$
\end{thm}
\begin{proof}
The projection is a simplicial map. It restricts to an embedding since
\eqref{eq:oembKs}
 is an (open) embedding.
\end{proof}
We can thus identify $K^s_\bullet(\mathcal A_{W|_U},Q)$ with 
$$\{A\in E(\Delta^\bullet, W)^0; (\exists A'\in\mathscr M(\Delta^\bullet,W)^0)\ d(A+A')=C_Q(A+A')\}$$
intersected with a suitable open subset of $E(\Delta^\bullet, W)^0$.

\section{Deformation retraction}\label{sec:def-retr}

In this section we shall prove that  $K^s_\bullet$ and $K^{big}_\bullet$ are equivalent as local Lie $\infty$-groupoids, namely we shall construct a local simplicial deformation retraction of $K^{big}_\bullet$ onto $K^s_\bullet$.

If $A_0\in\Omega_r(\Delta^\bullet,W|_U)^{MC}$ and 
$G\in \Omega_r(\Delta^\bullet,W)^{0,cl,p}$
so that $d(s_\bullet G)=G$,
 and if the $C^0$-norm of $s_\bullet G$ is small enough, then Theorem \ref{thm:htopy2} gives us a form 
$$\tilde A\in\Omega_r(\Delta^\bullet\times I,W|_U)^{MC}$$
 such that $\tilde A|_0=A_0$ and $\tilde A_v:= i_{\partial_t} \tilde A=q^*s_\bullet G$,
where $q:\Delta^\bullet\times I\to \Delta^\bullet$ is the projection. It is given by Equation \eqref{eq:dotAh2} with $H=q^*s_\bullet G$, i.e.\ by
$$\frac{d}{dt}\tilde A_h=G+(s_\bullet G^i)\frac{\partial C_Q}{\partial \xi^i}(\tilde A_h).$$
Let us use the notation $\tilde A(A_0,G)$ for the form $\tilde A$ constructed in this way.

\begin{prop}
There is an open neighbourhood $\mathcal V_\bullet$ of $U\times\{0\}$ in 
$$\Omega_r(\Delta^\bullet, W|_U)^{MC,s}\times\Omega_r(\Delta^\bullet,W)^{0,cl,p}$$
 such that the map $(A_0,G)\mapsto \tilde A(A_0, G)|_{t=1}$ is an open embedding $\psi:\mathcal V_\bullet\to\Omega_r(\Delta^\bullet, W|_U)^{MC}$.
\end{prop}
\begin{proof}
Since $\psi$ restricts to the identity on $U\times\{0\}$, it's enough to show that for each $A_c\in U$ the tangent map $\psi_{lin}:=T_{(A_c,0)}\psi$ is invertible. We have
$$\psi_{lin}(A_0,G)= A_0+G+(s_\bullet G^i)\frac{\partial C_Q}{\partial \xi^i}(A_c)$$
(for $A_0$ tangent to $\Omega_r(\Delta^\bullet, W|_U)^{MC,s}$ at $A_c$ and $G\in\Omega_r(\Delta^\bullet,W)^{0,cl,p}$). 

Let us recall that the projection $\pi_\bullet$ w.r.t.\ $(C\otimes W)^0$ (where $C=D(\Delta^\bullet)\oplus\mathscr M(\Delta^\bullet)$) gives us  isomorphisms 
$$T_{A_c}\Omega_r(\Delta^\bullet, W|_U)^{MC,s}\cong E(\Delta^\bullet,W)^{0,cl}$$ and $$T_{A_c}\Omega_r(\Delta^\bullet, W|_U)^{MC}\cong\Omega_r(\Delta^\bullet, W)^{0,cl}=E(\Delta^\bullet,W)^{0,cl}\oplus\Omega_r(\Delta^\bullet, W)^{0,cl,p}.$$
Since the term $(s_\bullet G^i)\frac{\partial C_Q}{\partial \xi^i}(A_c)\in\mathscr M(\Delta^\bullet, W)^{0}$ is removed by this projection, we see that $\psi_{lin}$ is indeed an isomorphism.
\end{proof}

\begin{thm}
There is an open neighbourhood $\mathcal W_\bullet$ of $U$ in 
$$\Omega_r(\Delta^\bullet, W|_U)^{MC}$$
which is a simplicial submanifold, and a (smooth) simplicial map
$$\Psi^s_\bullet:\mathcal W_\bullet\to\Omega_r(\Delta^\bullet\times I, W|_U)^{MC}$$
with these properties:
\begin{enumerate}
\item $\Psi^s_\bullet(A)|_{t=1}=A$ (for every $A\in\mathcal W_\bullet$)
\item $\Psi^s_\bullet(A)|_{t=0}\in \Omega_r(\Delta^\bullet, W|_U)^{MC,s}$
\item if $A\in\Omega_r(\Delta^\bullet, W|_U)^{MC,s}$ then $\Psi^s_\bullet(A)=q^*A$ where $q:\Delta^\bullet\times I\to\Delta^\bullet$ is the projection.
\end{enumerate}
\end{thm}

\begin{figure}[h!tb]
\center
\begin{tikzpicture}
\filldraw[fill=blue!20]
node[anchor=west,left,pos=0]{$\Delta^n\times \{t=0\}$} (0,0) --  node[anchor=west,left,pos=0.5]{$\Delta^n\times I$} node[anchor=west,left,pos=1]{$\Delta^n\times \{t=1\}$} (0,3) -- (1,2.5) -- (1,-0.5) -- (0,0);
\filldraw[fill=blue!20]
(2,0) -- (2,3) -- (1,2.5) -- (1,-0.5) -- (2,0);
\filldraw[fill=blue!10]
(0,3) -- (2,3) -- (1,2.5);
\draw[dashed] (0,0) -- (2,0);
\draw[->] (2,-0.25) to[out=270,in=180] (3,-0.5) node[anchor=east,right]{$\Psi^s_n 
(A)|_{t=0}\in\Omega(\Delta^n,W)^{MC,s}$};
\draw[->] (1.5,1.5) to[out=270, in=180] (3,1) node[anchor=east,right]{$\Psi^s_n 
(A)\in\Omega(\Delta^n\times I,W)^{MC}$};
\draw[->] (1,2.75) to[out=90, in=180] (3,3.5) node[anchor=east,right]{$A=\Psi^s_n 
(A)|_{t=1}\in\Omega(\Delta^n,W)^{MC}$};
\end{tikzpicture}
\caption{Deformation retraction}
\label{fig:dretr}
\end{figure}
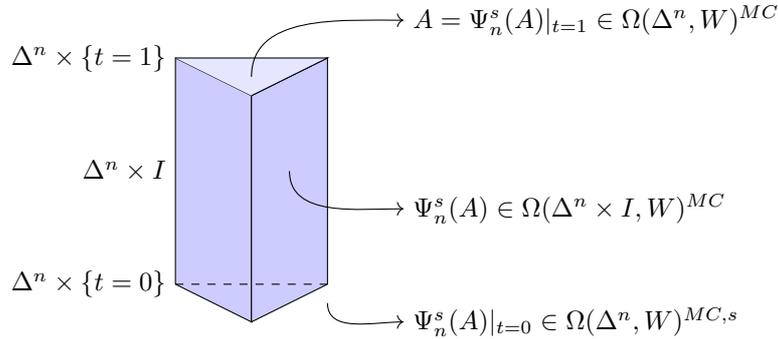

\begin{proof}
We set $\mathcal W_\bullet:=\psi(\mathcal V_\bullet)$ (we might then have to decrease $\mathcal W_\bullet$ to ensure that it is a simplicial submanifold), and $\Psi^s_\bullet(A):=\tilde A(\psi^{-1}(A))$.
\end{proof}

Let us recall that a simplicial homotopy is a simplicial map $X_\bullet\times {\mathrm I}_\bullet\to Y_\bullet$ where the simplicial set ${\mathrm I}_\bullet$ (the simplicial interval) is the set of non-decreasing maps $f:\{0,\dots,\bullet\}\to\{0,1\}.$ We can now formulate the main result of this section.

\begin{thm}
There is a local simplicial deformation retraction
$$R^s_\bullet:\Omega_r(\Delta^\bullet, W|_U)^{MC}\times {\mathrm I}_\bullet\to\Omega_r(\Delta^\bullet, W|_U)^{MC}$$
of $\Omega_r(\Delta^\bullet, W|_U)^{MC}$ to $\Omega_r(\Delta^\bullet, W|_U)^{MC,s}\subset\Omega_r(\Delta^\bullet, W|_U)^{MC}$,
where ``local'' means that it is defined on an open neighbourhood of $U$ in $\Omega_r(\Delta^\bullet, W|_U)^{MC}$.
\end{thm}

\begin{proof}
If $f\in {\mathrm I}_\bullet$, $f:\{0,\dots,\bullet\}\to\{0,1\}$, let $g:\Delta^\bullet\to\Delta^\bullet\times I$ be the affine map given on the vertices by $g(v_i)=(v_i,1-f(i))$ ($i=0,\dots,\bullet$). Given $A\in\mathcal W_\bullet\subset\Omega(\Delta^\bullet, W|_U)^{MC}$ we set $R^s_\bullet(A,f):=g^*\Psi^s_\bullet(A)$.
\end{proof}

\section{Functoriality and naturality}\label{sec:funct}
If $\phi:V\rsa V'$ is a morphism of NQ-manifolds, i.e.\ if we have a morphism of dg algebras $\phi^*:\mathcal A_{V'}\to\mathcal A_V$, by composition we get a morphism of Banach simplicial manifolds
$$(\phi_{*})_{\bullet}:\Omega_r(\Delta^\bullet, V)^{MC}\to\Omega_r(\Delta^\bullet, V')^{MC}$$
(or of Fr\'echet simplicial manifolds if we remove the subscript $r$). In this way $V\mapsto \Omega_r(\Delta^\bullet, V)^{MC}$ is a functor from the category of NQ-manifolds to the category of Banach simplicial manifolds.

The situation is more complicated when we consider the finite-dimensional simplicial manifolds 
$K^s_\bullet(W|_U):=\Omega_r(\Delta^\bullet, W|_U)^{MC,s}$. Let us first define the category in which we shall be working. 

\begin{defn}
If $X_\bullet$ and $Y_\bullet$ are simplicial manifold, a \emph{local homomorphism} $X_\bullet\to Y_\bullet$ is a smooth simplicial map $U_\bullet\to Y_\bullet$, where $U_\bullet\subset X_\bullet$ is an open simplicial submanifold containing all fully degenerate simplices. Two local homomorphisms are declared equal if they coincide on some $U_\bullet$.
\end{defn}

Let 
$$i^s_\bullet:\Omega_r(\Delta^\bullet, W|_U)^{MC,s}\to\Omega_r(\Delta^\bullet, W|_U)^{MC}$$
be the inclusion, and 
$$p^s_\bullet:\Omega_r(\Delta^\bullet, W|_U)^{MC}\to\Omega_r(\Delta^\bullet, W|_U)^{MC,s}$$
the projection given by $p^s_\bullet(A)=\Psi^s_\bullet(A)|_{t=0}$ ($p^s_\bullet$ is a local homomorphism).

If $\phi:W|_U\rsa W'|_{U'}$ is a map of NQ manifolds, i.e.\ if we have a morphism $\phi^*:\mathcal A_{W'|_{U'}}\to \mathcal A_{W|_{U}}$ of dg algebras, we get a local morphism
$$K^s_\bullet(\phi):\Omega_r(\Delta^\bullet, W|_U)^{MC,s}\to\Omega_r(\Delta^\bullet, W'|_{U'})^{MC,s}$$
of local Lie $\ell$-groupoids, defined as the composition
$$\Omega_r(\Delta^\bullet, W|_U)^{MC,s}\xrightarrow{i_s}\Omega_r(\Delta^\bullet, W|_U)^{MC}\xrightarrow{(\phi_{*})_{\bullet}}\Omega_r(\Delta^\bullet, W'|_{U'})^{MC}\xrightarrow{p_s}\Omega_r(\Delta^\bullet, W'|_{U'})^{MC,s}.$$

 Let us observe that 
$$K^s_\bullet(\phi\circ\phi')\neq K^s_\bullet(\phi)\circ K^s_\bullet(\phi')$$ 
in general, i.e.\ $K^s_\bullet$ is not a functor. It is, however, a ``functor up to homotopy'', i.e.\ a homotopy coherent diagram in the sense of Vogt \cite{Vogt}, or, using a more recent terminology,  a quasi-functor. Let us  describe this quasi-functor explicitly. An $n$-simplex $f$ in the nerve of the category of NQ manifolds of the type $W|_U$ is a chain of composable morphisms 
$$f:=\Big(W_0|_{U_0}\xrsa{\phi_0} W_1|_{U_1}\xrsa{\phi_1} \cdots \xrsa{\phi_{n-1}} W_n|_{U_n}\Big).$$ 
Combining the maps $(\phi_{i*})_\bullet$, $i=0,\ldots, n-1$ with the local simplicial deformation retractions $R^s_\bullet$ we obtain simplicial maps 
$$K^s_\bullet(f):K^s_\bullet(W_0|_{U_0})\times {\mathrm I}_\bullet^{n-1}\to K^s_\bullet(W_n|_{U_n})$$ 
defined (in the case of $n=3$, as an illustration) as the composition
$$
\begin{tikzcd}
 K^{s}_\bullet(W_0|_{U_0})\times({\mathrm I}_\bullet)^2\arrow{d}[swap]{i^s_\bullet\times id} & K^{big}_\bullet(W_1|_{U_1})\times({\mathrm I}_\bullet)^2\arrow{d}{R^s_\bullet\times id} & K^{big}_\bullet(W_2|_{U_2})\times {\mathrm I}_\bullet\arrow{d}{R^s_\bullet} & K^{big}_\bullet(W_3|_{U_3})\arrow{d}{p^s_\bullet}   \\
 K^{big}_\bullet(W_0|_{U_0})\times({\mathrm I}_\bullet)^2\arrow{ur}[description]{(\phi_{0*})_\bullet\times id}  & K^{big}_\bullet(W_1|_{U_1})\times {\mathrm I}_\bullet\arrow{ur}[description]{(\phi_{1*})_\bullet\times id} & K^{big}_\bullet(W_2|_{U_2})\arrow{ur}[description]{(\phi_{2*})_\bullet} & K^{s}_\bullet(W_3|_{U_3})
\end{tikzcd}
$$
By construction, they define a homotopy coherent diagram. To explain what it means, let 
us use Vogt's notation for $K^s_\bullet(f)$
$$K^s_\bullet(\phi_0,t_1,\phi_1,t_2,\dots,t_{n-1},\phi_{n-1}):K^s_\bullet(W_0|_{U_0})\to K^s_\bullet(W_n|_{U_n})$$
 where $t_i\in {\mathrm I}_\bullet$, i.e.\ $t_i$'s are non-decreasing maps $\{0,\dots,\bullet\}\to\{0,1\}$. Then
\begin{multline}\label{eq:verybig}
K^s_\bullet(\phi_0,t_1,\dots,t_{n-1},\phi_{n-1})=\\
=\begin{cases}
K^s_\bullet(\phi_1,t_2,\dots,t_{n-1},\phi_{n-1}) & \text{if }\phi_0=id\\
K^s_\bullet(\phi_0,t_1,\dots,\max(t_{i},t_{i+1}),\dots,t_{n-1},\phi_{n-1}) & \text{if }\phi_i=id,\ 0<i<n-1\\
K^s_\bullet(\phi_0,t_1,\dots,t_{n-2},\phi_{n-2}) & \text{if }\phi_{n-1}=id\\
K^s_\bullet(\phi_0,t_1,\dots,\phi_i\circ \phi_{i-1},\dots,t_{n-1},\phi_{n-1}) & \text{if }t_i=0\text{ identically}\\
K^s_\bullet(\phi_i,t_{i+1},\dots,\phi_{n-1})\circ K^s_\bullet(\phi_0,t_1,\dots,\phi_{i-1}) & \text{if }t_i=1\text{ identically}
\end{cases}
\end{multline}
which is Vogt's definition of a homotopy coherent diagram \cite{Vogt}.\footnote{
More compactly, a homotopy coherent diagram can be described as follows (an introduction to the subject can be found in Porter \cite{Porter}). If $\mathcal C$ is a category (in our case the category of NQ-manifolds of the form $W|_U$) and if $\mathcal D$ is a simplicially enriched category  (in our case the category of simplicial manifolds with local morphisms) then a homotopy coherent diagram from $\mathcal C$ to $\mathcal D$ is a simplicially enriched functor $S(\mathcal C)\to\mathcal D$, where $S(\mathcal C)$ is the simplicially enriched category  defined as the free simplicial resolution of $\mathcal C$, introduced by Dwyer and Kan \cite{DwyerKan}.}
We thus have the following result.

\begin{thm}
$K^s_\bullet$ is a homotopy coherent diagram from the category of NQ manifolds of type $W|_U$ to the simplicially enriched category of local simplicial manifolds with the simplicial enrichment given by 
$$\text{Hom}(L,M)_n:=\text{Hom}(L\times\Delta[n],M)$$ 
where $\Delta[n]$ is the simplicial set representing the $n$-simplex.
Here ``local'' means that morphisms are defined to be local homomorphisms of simplicial manifolds.
\end{thm}

For a general NQ manifold given by a negatively graded vector bundle $V\to M$ (not of the form $W|_U$) one can apply the quasi-functor $K^s_\bullet$ on the nerve $N_\bullet(\mathsf{U})$ of a good cover $\mathsf{U}$ of $M$. We can say a bit more in this situation: the homomorphisms $K^s_\bullet(\phi)$ on overlaps are actually isomorphisms.

Indeed, for any dg-map $\phi:W|_U\rsa W'|_{U'}$
the linearization $(\phi_*)_{\bullet,\text{lin}}$ of  $(\phi_*)_\bullet$ commutes with $s_\bullet, p_\bullet$. Therefore the linearization $(\phi_*)^s_{\bullet,\text{lin}}$ of  $(\phi_*)^s_\bullet$ is functorial, i.e.\ if $\phi':W'|_{U'}\rsa W''|_{U''}$ is another map of NQ manifolds then $(\phi'_*\circ\phi_*)^s_{\bullet,\text{lin}}=(\phi'_*)^s_{\bullet,\text{lin}}\circ (\phi_*)^s_{\bullet,\text{lin}}$. For $\text{id}:W|_U\rsa W|_U$ we have $(\text{id}_*)_\bullet^s=\text{id}_{\Omega_r(\Delta^\bullet,W|_U)^{MC,s}}$. Together with functoriality this implies that if $\phi$ is an isomorphism of dg manifolds then $K^s_\bullet(\phi)$ is an isomorphism of local Lie $\ell$-groupoids.

%\section{Integration of Courant algebroids and other symplectic NQ manifolds}

\section{Integration of (pre)symplectic forms}\label{sec:sympl}

One of the most remarkable facts of Poisson geometry is that if $M$ is a Poisson manifold then $T^*M$ is a Lie algebroid (the corresponding differential on $\Gamma(\bigwedge TM)$ is given by $[\pi,\cdot]$ where $\pi$ is the Poisson structure and $[,]$ the Schouten bracket) and that the corresponding (local) Lie groupoid is symplectic. In \cite{Severa} it was suggested that Courant algebroids should be integrated to (local) 2-symplectic 2-groupoids, and more generally, NQ-manifolds with $Q$-invariant symplectic form of degree $k$ should integrate to (local) $k$-symplectic $\ell$-groupoids (where $k\geq\ell$ and $k=\ell$ whenever the base $M$ is non-trivial). In this section we shall see that it is indeed the case. Let us remark that in the special case of exact Courant algebroids the integration procedure was studied in detail in \cite{MT}.

\begin{defn}
A NQ-manifold $Z$ is $k$-symplectic (resp.\ $k$-presymplectic) if it carries a $Q$-invariant symplectic (resp.\ closed) 2-form $\varpi$ of degree $k$ w.r.t.\ the grading on $Z$.
\end{defn}

As observed in \cite{Severa}, if $\varpi$ is symplectic then $k\geq\ell$ and $k=\ell$ whenever $\dim M\geq1$. Another observation from \cite{Severa} is that any 1-symplectic NQ-manifold is naturally of the form $T^*[1]M$ and the differential $Q$ is given by a Poisson structure on $M$, and that any 2-symplectic NQ-manifold is equivalent to a Courant algebroid. (Courant algebroids were introduced by Liu, Weinstein and Xu in \cite{LWX} and their connection with 2-symplectic NQ-manifolds is explained in detail in Roytenberg \cite{Roytenberg}.) 

To discuss symplectic forms on Lie $\ell$-groupoids, 
let us recall that if $K_\bullet$ is a simplicial manifold then $\bigoplus_{m,n}\Omega^m(K_n)$ is a bicomplex. The first differential is de Rham's $d$ and the second differential $\delta$ is given by the simplicial structure
$$\delta\alpha:=\sum_{p=0}^{n+1}(-1)^p d_p^*\alpha\in\Omega^m(K_{n+1}) \text{ for } \alpha\in\Omega^m(K_n)$$
where $d_p:K_{n+1}\to K_n$ are the face maps.

A Lie groupoid is called \emph{symplectic} if its nerve $K_\bullet$ is endowed with a symplectic form $\omega\in\Omega^2(K_1)$ such that $\delta\omega=0$.

\begin{defn}
A (local) Lie $\ell$-groupoid $K_\bullet$ is \emph{strictly $k$-symplectic} (resp.\ \emph{presymplectic}) if it is endowed with a symplectic (resp.\ closed) 2-form $\omega\in\Omega^2(K_k)$ such that $\delta\omega=0$.
\end{defn}

Let now $Z$ be a $k$-presymplectic NQ-manifold. We shall construct a closed 2-form $\omega^{big}$ on $K^{big}_k(Z)$ satisfying $\delta\omega^{big}=0$. Moreover, if $\varpi$ is symplectic and if $Z$ is of the form $W|_U$ (or if $W|_U$ is a local piece of $Z$) we shall see that the 2-from $\omega^{big}$ restricts to a symplectic form on $K^{s}_k(W|_U)$, i.e.\ that $K^{s}_\bullet(W|_U)$ is a local $k$-symplectic $\ell$-groupoid.

Suppose that $N$ is a compact oriented manifold (possibly with corners) and $f:T[1]N\rsa Z$ a NQ-map. Following the AKSZ construction \cite{AKSZ}, if $u,v\in\Gamma(f^*TZ)$, let us define
$$\omega_{N,f}(u,v):=\int_N(f^!\varpi)(u,v)$$
where $f^!\varpi$ denotes $\varpi$ pulled back to a pairing on  $f^*TZ$
($(f^!\varpi)(u,v)$ is a function on $T[1]N$, i.e.\ a differential form on $N$, and we integrate its top part over $N$). If $Z=W|_U$ then $u,v\in\Omega(N,W)$. In this case, if $\varpi=\varpi_{ij}(\xi)d\xi^i d\xi^j$ ($\varpi_{ij}(\xi)\in\mathcal A_{W|_U}$) and  $A\in\Omega(N,W|_U)^{MC}$ then 
$$\omega_{N,A}(u,v)=\int_N\varpi_{ij}(A)u^iv^j.$$

The differentials $d$ on $T[1]N$ and $Q$ on $Z$ turn $\Gamma(f^*TZ)$ to a differential graded module over $\Omega(N)=C^\infty(T[1]N)$, with a differential $d_{tot}$. When $Z=W|_U$ and thus $\Gamma(f^*TZ)=\Omega(N,W)$, the differential $d_{tot}$ can be computed as
$$d_{tot}u=du-u^i\frac{\partial C_Q}{\partial\xi ^i}(A)\quad\forall u\in\Omega(N,W).$$

The $Q$-invariance of $\varpi$ implies the identity (with signs unimportant for what follows)
\begin{equation}\label{eq:omega-dtot}
d\bigl((f^!\varpi)(u,v)\bigr)=\pm(f^!\varpi)(d_{tot}u,v)\pm(f^!\varpi)(u,d_{tot}v)\in\Omega(N).
\end{equation}
As a consequence, we get the following result (Lemma 3 of \cite{Severa}).

\begin{prop}\label{prop:omega-omega}
Let $i_p:\Delta^{n-1}\to\Delta^n$ be the inclusion of the $p$-th face of $\Delta^n$. Then
\begin{equation}\label{eq:omega-omega}
\sum_{p=0}^n(-1)^p\omega_{\Delta^{n-1},i_p^*f}(i_p^*u,i_p^*v)=\pm\omega_{\Delta^n,f}(d_{tot}u,v)\pm\omega_{\Delta^n,f}(u,d_{tot}v)
\end{equation}
\end{prop}
\begin{proof}
 The claim follows from the Stokes theorem applied to \eqref{eq:omega-dtot} with $N=\Delta^n$.
\end{proof}

The tangent space of $K^{big}_k(Z)$ at $f:T[1]\Delta^k\to Z$ is $\Gamma(f^*TZ)^{0,cl}$.
Following \cite{AKSZ} we now define a closed 2-form $\omega^{big}$ on the manifold $K^{big}_k(Z)$  by 
$$\omega^{big}(u,v):=\omega_{\Delta^k,f}(u,v).$$
As a consequence of Proposition \ref{prop:omega-omega} we get the following result.

\begin{thm}
The closed 2-form $\omega^{big}\in\Omega^2(K^{big}_k(Z))$ satisfies $\delta\omega^{big}=0$.
\end{thm}
\begin{proof}
Let us use Equation \eqref{eq:omega-omega} when $u,v\in T_f K^{big}_k(Z)=\Gamma(f^*TZ)^{0,cl}$: the RHS vanishes as $d_{tot}u=d_{tot}v=0$, and the LHS is $(\delta\omega^{big})(u,v)$. As a result $\delta\omega^{big}=0$.
\end{proof}

Suppose now that $Z$ is of the form $W|_U$. Let us define the closed 2-form 
$$\omega^s_{W|_U}\in\Omega^2(K^s_k(W|_U))$$
 as the restriction of $\omega^{big}$ to the simplicial submanifold $K^s_k(W|_U)\subset K^{big}_k(W|_U)$. As $\omega^{big}$ it satisfies the relation
$$\delta\omega^s_{W|_U}=0.$$

\begin{thm}
If the closed 2-form $\varpi$ on $W|_U$ is symplectic then $\omega^s_{W|_U}\in\Omega^2(K^s_k(W|_U))$ is symplectic on an open neighbourhood of the totally degenerate simplicies $U\subset K^s_k(W|_U)$. The local Lie $\ell$-groupoid $K^s_\bullet(W|_U)$ is thus strictly $k$-symplectic (after we possibly replace it with an open neighbourhood of the totally degenerate simplices).
\end{thm}
\begin{proof}
It is enough to prove that $\omega^s_{W|_U}$ is non-degenerate at the points of $U\subset K^s_k(W|_U)$. If $x\in U$, the linearization of $Q$ at $x$ gives us a differential $Q_{lin}:W\to W$ making $W$ to a cochain complex; explicitly, if $w\in W$,
$$Q_{lin}w=-w^i\frac{\partial C_Q}{\partial\xi ^i}(x).$$
 The tangent space $T_x K^s_\bullet(W|_U)$ is
 $$T_x K^s_\bullet(W|_U)=E(\Delta^\bullet,W)^{0,cl_{tot}}$$
where the subscript $cl_{tot}$ means closed w.r.t.\ the total differential $d+Q_{lin}$.

Let us introduce Grassmann parameters $\epsilon_0,\dots,\epsilon_n$ of degree 1 and consider the graded vector space $\tilde E_n:=\bigwedge(\epsilon_0,\dots,\epsilon_n)$. On $\tilde E_n$ there is a differential 
$$d=\sum \epsilon_i$$
of degree 1, and a differential
$$\partial=\sum\frac{\partial}{\partial\epsilon_i}$$
of degree -1, and
$$d\partial+\partial d=n+1.$$
Both $d$ and $\partial$ are thus acyclic and $\tilde E_n$ is the direct sum of its $d$-closed part and it $\partial$-closed part.

Moreover, we have a morphism of chain complexes
$$\chi:(\tilde E_n[1],d)\to (E(\Delta^n),d),\ \epsilon_{i_1}\dots\epsilon_{i_m}\mapsto\omega_{i_1\dots i_m},\ 1\mapsto 0$$
which is an isomorphism with the exception of degree $-1$. We can use $\chi$ to identify $T_x K^s_n(W|_U)=E(\Delta^n,W)^{0,cl_{tot}}$ with $(\tilde E_n\otimes W)^{1,cl_{tot}}$.

Let us use the non-degenerate pairing $\tilde E_n\otimes\tilde E_n\to\R$
$$\langle\sigma,\tau\rangle:=\text{the coefficient of $\epsilon_0\dots\epsilon_n$ in $\sigma\tau$}$$
and the pairing $\tilde E_n\otimes\tilde E_n\to\R$
$$(\sigma,\tau)=\int_{\Delta^n}\chi(\sigma)\chi(\tau).$$
A straightforward calculation shows that
$$(\sigma,\tau)=\frac{(i-1)!(j-1)!}{(n+1)!}\langle\sigma,\partial\tau\rangle\quad\forall\sigma\in\tilde E_n^i,\tau\in\tilde E_n^j (i,j\geq1).$$
 As a consequence, the kernel of $(,)$ is the $\partial$-closed  part of $\tilde E_n$.

We can now prove that $\omega^s_{W|_U}$ is non-degenerate at $x$. The symplectic form $\varpi$ at $x$ gives us a non-degenerate pairing $\varpi_x:W\otimes W\to\R$ of degree $k$, and $\omega^s_{W|_U}$ on 
$$T_x K^s_k(W|_U)\cong(\tilde E_k\otimes W)^{1,cl_{tot}}$$
is, by definition, the restriction of $(,)\otimes\varpi_x$ to $(\tilde E_k\otimes W)^{1,cl_{tot}}\subset\tilde E_k\otimes W$. The kernel of $(,)\otimes\varpi_x$ is the $\delta$-closed part of $\tilde E_k\otimes W$ and $(\tilde E_k\otimes W)^{cl_{tot}}$ is its complement (as $d_{tot}\partial+\partial d_{tot}=k+1$), hence $(,)\otimes\varpi_x$ is non-degenerate on $(\tilde E_k\otimes W)^{1,cl_{tot}}$, as we wanted to show.
\end{proof}

\section{$A_\infty$-functoriality of $\omega^s$}\label{sec:sympl-funct}

If $Y$ is a simplicial set and $\tau\in Y_n$, let $\hat\tau:\Delta[n]\to Y$ be the  morphism sending the non-degenerate $n$-simplex of $\Delta[n]$ to $\tau$.
If $X$ is another simplicial set and if $\sigma\in(\Delta[n]\times X)_N$ (for some $N\in\mathbb N$), let
$$\sigma^\sharp:Y_n\to Y_N\times X_N$$
be the map defined via
$$\sigma^\sharp(\tau)=(\hat\tau\times id_X)(\sigma).$$

If $K$ is a simplicial manifold, we thus get the map (where we understand $X_N$ as discrete)
$$\sigma^*:=(\sigma^\sharp)^*:\Omega(K_N\times X_N)=\Omega(K_N)\times X_N\to\Omega(K_n).$$
More generally, if $c=\sum_i a_i\sigma_i$ ($a_i\in\R$) is a $N$-chain in $\Delta[n]\times X$, we set
$$c^*:=\sum_i a_i \sigma_i^*.$$
By construction we have
\begin{equation}\label{eq:c-star-del}
(\partial c)^*\alpha=\delta(c^*\alpha)
\end{equation}
for every $\alpha\in\Omega(K_n)$.

Let us now consider the special case $X=\mathrm I^m$ and the chain 
$$c_{m,n}=[I^m\times\Delta^n]\in C_{m+n}(\Delta[n]\times\mathrm I^m)$$
 giving the fundamental class (rel boundary) of $I^m\times\Delta^n$; it is the signed sum of all non-degenerate $m+n$-simplices of $\Delta[n]\times\mathrm I^m$ with the signs given by comparing with the orientation of the space $I^m\times\Delta^n$.
Let us use the notation
$$\mathscr I^m:=c_{m,n}^*:\Omega(K_{n+m}\times(\mathrm I_{n+m})^m)\to\Omega(K_n).$$

\begin{prop}\label{prop:i}
For any $\alpha\in\Omega\bigl(K_{n+m}\times(\mathrm I_{n+m})^m\bigr)$ we have
$$\bigl(\delta\mathscr I^m-(-1)^m\mathscr I^m\delta\bigr)\alpha=\sum_{r=1}^m (-1)^{r-1}\Bigl(\mathscr I^{m-1}(\alpha|_{1_r})-\mathscr I^{m-1}(\alpha|_{0_r})\Bigr)\in\Omega(K_{n+1})$$
where $\alpha|_{0_r},\alpha|_{1_r}\in\Omega\bigl(K_n\times(\mathrm I_n)^{m-1}\bigr)$ is obtained from $\alpha$ by restricting the $r$'th $\mathrm I_n$ to $0$ and to $1$ respectively.
\end{prop}
\begin{proof}
The boundary of $c_{m,n}=[I^m\times\Delta^n]\in C_{m+n}(\Delta[n]\times\mathrm I^m)$ is
$$\partial c_{m,n}=\partial[I^m]\times[\Delta]^n+(-1)^m[I^m]\times\partial[\Delta^n],$$
so Equation \eqref{eq:c-star-del} gives us the identity (for every $\alpha\in\Omega(K_{m+n}\times(\mathrm I_{m+n})^m)$)
$$\delta(c_{m,n}^*\alpha)=\sum_{r=1}^m (-1)^{r-1}\Bigl(c_{m-1,n}^*(\alpha|_{1_r})-c_{m-1,n}^*(\alpha|_{0_r})\Bigr)+(-1)^m c_{m,n-1}^*\delta\alpha$$
which is the identity we wanted to prove.
\end{proof}

We can now deal with the problem of functoriality of $\omega^s$. Let us consider the category $\mathcal C_{\varpi,k}$ of $k$-presymplectic NQ-manifolds of the form $W|_U$. Morphisms of this category are NQ-maps $\phi:W|_U\rsa W'|_{U'}$ such that $\phi^*\varpi_{W'|_{U'}}=\varpi_{W|_U}$.

For every chain of morphisms of $\mathcal C_{\varpi,k}$ (i.e.\ for every simplex of the nerve of $\mathcal C_{\varpi,k}$)
$$f:=\Big(W_0|_{U_0}\xrsa{\phi_0} W_1|_{U_1}\xrsa{\phi_1} \cdots \xrsa{\phi_{n-1}} W_n|_{U_n}\Big)$$
let us define $\omega^s_{W_0|_{U_0}}(f)\in\Omega^{2,cl}\bigl(K^s_{k-n}(W_0|_{U_0})\bigr)$ as $\mathscr I^n$ of the pullback of $\omega^{big}_{W_n|_{U_n}}$ via
the composition
 (with $n=2$ as an illustration)
\begin{equation}\label{eq:snake-omega}
\begin{tikzcd}
 K^{s}_\bullet(W_0|_{U_0})\times({\mathrm I}_\bullet)^2\arrow{d}[swap]{i^s_\bullet\times id} & K^{big}_\bullet(W_1|_{U_1})\times({\mathrm I}_\bullet)^2\arrow{d}{R^s_\bullet\times id} & K^{big}_\bullet(W_2|_{U_2})\times {\mathrm I}_\bullet\arrow{d}{R^s_\bullet}   \\
 K^{big}_\bullet(W_0|_{U_0})\times({\mathrm I}_\bullet)^2\arrow{ur}[description]{(\phi_{0*})_\bullet\times id}  & K^{big}_\bullet(W_1|_{U_1})\times {\mathrm I}_\bullet\arrow{ur}[description]{(\phi_{1*})_\bullet\times id} & K^{big}_\bullet(W_2|_{U_2})
\end{tikzcd}
\end{equation}

\begin{thm}\label{thm:del-omega-s}
The closed 2-forms $\omega^s(f)$ satisfy the identities
\begin{multline}\label{eq:del-omega-s}
\delta\omega^s_{W_0|_{U_0}}(\phi_0,\phi_1,\dots,\phi_{n-1})=\\
=\omega^s_{W_0|_{U_0}}(\phi_1\circ\phi_0,\phi_2,\dots,\phi_{n-1})-\omega^s_{W_0|_{U_0}}(\phi_0,\phi_2\circ\phi_1,\phi_3,\dots,\phi_{n-1})+{}\\
\dots+(-1)^{n-1}\omega^s_{W_0|_{U_0}}(\phi_0,\phi_1,\dots,\phi_{n-2})-{}\\
-\sum_{i=1}^n(-1)^{i-1}\mathscr I^{i-1}\bigl(K^s(\phi_0,\dots,\phi_{i-1})^*\omega^s_{W_{i}|_{U_{i}}}(\phi_i,\dots,\phi_{n-1})\bigr),
\end{multline}
\begin{equation}\label{eq:omega-s-len}
\omega^s_{W_0|_{U_0}}(\phi_0,\phi_1,\dots,\phi_{n-1})=0\text{ for }n\geq k,
\end{equation}
\begin{equation}\label{eq:omega-s-deg}
\omega^s_{W_0|_{U_0}}(\phi_0,\phi_1,\dots,\phi_{n-1})=0\text{ if some }\phi_i=id.
\end{equation}
\end{thm}
\begin{proof}
Let $q_\bullet(\phi_0,\dots,\phi_{n-1}):K^s_\bullet(W_0|_{U_0})\times(\mathrm I_\bullet)^n\to K^{big}_\bullet(W_n|_{U_n})$ be the simplicial map given by the composition of the snake \eqref{eq:snake-omega}, so that 
$$\omega^s_{W_0|_{U_0}}(\phi_0,\phi_1,\dots,\phi_{n-1})=\mathscr I^n\bigl(q_\bullet(\phi_0,\phi_1,\dots,\phi_{n-1})^*\omega^{big}\bigr).$$
 For convenience, let us use the notation
$$q_\bullet(\phi_0,t_1,\phi_1,t_2,\dots,t_{n-1},\phi_{n-1},t_n):K^s_\bullet(W_0|_{U_0})\to K^{big}_\bullet(W_n|_{U_n})$$
with $t_i\in\mathrm I_\bullet$. Similarly to \eqref{eq:verybig} we have the identities
\begin{multline}\label{eq:yetbigger}
q_\bullet(\phi_0,t_1,\dots,t_{n-1},\phi_{n-1},t_n)=\\
=\begin{cases}
q_\bullet(\phi_1,t_2,\dots,t_{n-1},\phi_{n-1},t_n) & \text{if }\phi_0=id\\
q_\bullet(\phi_0,t_1,\dots,\max(t_{i},t_{i+1}),\dots,\phi_{n-1},t_n) & \text{if }\phi_i=id,\ 1\leq i\leq n-1\\
q_\bullet(\phi_0,t_1,\dots,\phi_i\circ \phi_{i-1},\dots,\phi_{n-1},t_n) & \text{if }t_i=0\text{ identically},\ 1\leq i\leq n-1\\
(\phi_{n-1})_*\circ q_\bullet(\phi_0,t_1,\dots,\phi_{n-2},t_{n-1}) & \text{if }t_n=0\text{ identically}\\
q_\bullet(\phi_i,t_{i+1},\dots,\phi_{n-1},t_n)\circ K^s_\bullet(\phi_0,t_1,\dots,\phi_{i-1}) & \text{if }t_i=1\text{ identically},\ 1\leq i\leq n.
\end{cases}
\end{multline}
Since $\delta\omega^{big}=0$, we have 
$$\delta\omega^s_{W_0|_{U_0}}(\phi_0,\phi_1,\dots,\phi_{n-1})=(\delta\mathscr I^n-(-1)^{n}\mathscr I^n\delta)\bigl( q_\bullet(\phi_0,\phi_1,\dots,\phi_{n-1})^*\omega^{big}\bigr)$$
and Proposition \ref{prop:i} and the last three cases of Equation \eqref{eq:yetbigger} give us Equation \eqref{eq:del-omega-s}. Equation \eqref{eq:omega-s-deg} follows from the first two cases of Equation \eqref{eq:yetbigger} and finally Equation \eqref{eq:omega-s-len} is obvious.
\end{proof}

\begin{rem}[Courant algebroids and Dirac structures]
When $k=1$ then, by \eqref{eq:omega-s-len},  $\omega^s_{W_0|_{U_0}}(\phi_0)=0$, and Equation \eqref{eq:del-omega-s} with $n=1$ becomes
$$\omega^s_{W_0|_{U_0}}=K^s(\phi_0)^*\omega^s_{W_1|_{U_1}}.$$
The first non-trivial case is thus $k=2$. Equation \eqref{eq:del-omega-s} for $n=1$ is
\begin{subequations}\label{eq:CA}
\begin{equation}\label{eq:CA1}
\delta\omega^s_{W_0|_{U_0}}(\phi_0)=\omega^s_{W_0|_{U_0}}-K^s(\phi_0)^*\omega^s_{W_1|_{U_1}}
\end{equation}
and for $n=2$  (using $\omega^s_{W_0|_{U_0}}(\phi_0,\phi_1)=0$)
\begin{equation}
\omega^s_{W_0|_{U_0}}(\phi_0)-\omega^s_{W_0|_{U_0}}(\phi_1\circ\phi_0)+K^s(\phi_0)^*\omega^s_{W_1|_{U_1}}(\phi_1)=\mathscr I K^s(\phi_0,\phi_1)^*\omega^s_{W_2|_{U_2}}.
\end{equation}
\end{subequations}

In particular if we have a Courant algebroid over a manifold $M$, if $Z$ is the corresponding 2-symplectic NQ-manifold (and thus $0\cdot Z=M$) and if $W_i|_{U_i}$ are (isomorphic to) local pieces of $Z$ and $\phi_j$'s are their identifications on the overlaps, Equations \eqref{eq:CA} show us in what sense the symplectic forms $\omega^s_{W_i|_{U_i}}$ agree up to a coherent homotopy.

If $Y\subset Z$ is a Lagrangian NQ-submanifold (i.e.\ a (generalized) Dirac structure of the Courant algebroid), if $W_0|_{U_0}$ is a local piece of $Y$, $W_1|_{U_1}$ a local piece of $Z$, and $\phi_0:W_0|_{U_0}\rsa W_1|_{U_1}$ the inclusion $Y\subset Z$, then we have $\omega^s_{W_0|_{U_0}}=0$ and Equation \eqref{eq:CA1} becomes
$$\delta\omega^s_{W_0|_{U_0}}(\phi_0)=-K^s(\phi_0)^*\omega^s_{W_1|_{U_1}}.$$
The closed 2-form $K^s(\phi_0)^*\omega^s_{W_1|_{U_1}}$ thus doesn't have to vanish, i.e.\ $K^s_k(\phi_0):K^s_k(W_0|_{U_0})\to K^s_k(W_1|_{U_1})$ is not necessarily a Lagrangian embedding, however $K^s(\phi_0)^*\omega^s_{W_1|_{U_1}}$  is  homotopic to zero.
\end{rem}

Theorem \ref{thm:del-omega-s} can be interpreted as $A_\infty$-functoriality.
Let $\mathcal C$ be the category of the NQ-manifolds of the form $W|_U$. Let $F(W|_U)$ denote the cochain complex
$$F(W|_U):=\bigl(\Omega^{2,cl}(K^s_\bullet(W|_U)),\delta\bigr)$$
and for 
$$f:=\Big(W_0|_{U_0}\xrsa{\phi_0} W_1|_{U_1}\xrsa{\phi_1} \cdots \xrsa{\phi_{n-1}} W_n|_{U_n}\Big)$$
let
$$F(f):=\mathscr I^{n-1}\circ K^s(\phi_0,\dots,\phi_{n-1})^*:F(W_n|_{U_n})\to F(W_0|_{U_0})$$
\begin{prop}\label{prop:F-A-inf}
$F$ is a (strictly unital) contravariant $A_\infty$-functor from $\mathcal C$ to the category of cochain complexes, i.e.\ $\deg F(\phi_0,\dots,\phi_{n-1})=1-n$, $F(id)=id$ and $F(\phi_0,\dots,\phi_{n-1})=0$ if $n\geq 2$ and some $\phi_i=id$, and finally
\begin{multline*}
\delta\circ F(\phi_0,\dots,\phi_{n-1})-(-1)^{1-n}F(\phi_0,\dots,\phi_{n-1})\circ\delta=\\
=\sum_{p=1}^{n-2} (-1)^p \bigl(F(\phi_0,\dots,\phi_{p-1})\circ F(\phi_p,\dots,\phi_{n-1})-F(\phi_0,\dots,\phi_p\circ\phi_{p-1},\dots,\phi_{n-1})\bigr).
\end{multline*}
\end{prop}
\begin{proof}
The claim follows from the properties of $K^s(f)$ listed in Equation \eqref{eq:verybig} and from the property of $\mathscr I^m$ given in Proposition \ref{prop:i}.
\end{proof}

Let now ${}^c\mathcal C_{\varpi,k}$ be the cocone category of $\mathcal C_{\varpi,k}$. By definition, it contains $\mathcal C_{\varpi,k}$ as a full subcategory, a unique additional  object $*$, and a unique additional morphism $X\to *$ for any $X\in \mathcal C_{\varpi,k}$. 
Let us extend $F$ from $\mathcal C$ to ${}^c\mathcal C_{\varpi,k}$ as follows: for any $f\in N_p{}^c\mathcal C_{\varpi,k}$ ($p=0,1,\dots$), 
$$f=(X_0\to X_1\to\dots\to X_p),$$ 
let
\begin{align*}
\tilde F(f)&=F(f)\text{ if }X_p\neq*\\
\tilde F(*)&=\R[-k]\\
\tilde F(f)&=\omega^s_{X_0}(X_0\to X_1\to\dots\to X_{p-1})\text{ if $p>0$ and }X_p=*,
\end{align*}
where $\omega^s_{X_0}(X_0\to\dots\to X_{p-1})\in\Omega^{2,cl}(K^s_{k-p}(X_0))$ is understood as a map $\R[-k]\to\Omega^{2,cl}(K^s_\bullet(X_0))$ of degree $-p$.

\begin{thm}
$\tilde F$ is a (strictly unital) contravariant $A_\infty$-functor from ${}^c\mathcal C_{\varpi,k}$ to the category of cochain complexes.
\end{thm}
\begin{proof}
The claim is a combination of Theorem \ref{thm:del-omega-s} and Proposition \ref{prop:F-A-inf}.
\end{proof}

\end{document}